\newcommand{\version}{}

\documentclass[a4paper,12pt]{amsart}

\usepackage[paper=a4paper,textwidth=400pt,textheight=620pt,includeheadfoot,twoside=True,centering]{geometry}

\usepackage[dvipdfm,svgnames]{xcolor}
\usepackage{graphics}

\usepackage{amsmath,amssymb, amsthm,amsxtra}

\usepackage[mathscr]{eucal}

\usepackage{bm}
\usepackage[all]{xy}

\usepackage{aliascnt}

\theoremstyle{plain}
\newtheorem{thm}{Theorem}[section]
\newtheorem*{thm*}{Theorem}

\newaliascnt{prop}{thm}
\newaliascnt{cor}{thm}
\newaliascnt{lem}{thm}
\newaliascnt{claim}{thm}
\newaliascnt{defn}{thm}
\newaliascnt{ques}{thm}
\newaliascnt{conj}{thm}
\newaliascnt{fact}{thm}
\newaliascnt{rem}{thm}
\newaliascnt{ex}{thm}
\newtheorem{prop}[prop]{Proposition}
\newtheorem{cor}[cor]{Corollary}
\newtheorem{lem}[lem]{Lemma}

\newtheorem*{prop*}{Proposition}
\newtheorem*{cor*}{Corollary}
\newtheorem*{lem*}{Lemma}
\newtheorem*{claim*}{Claim}
\theoremstyle{definition}
\newtheorem{defn}[defn]{Definition}

\newtheorem*{defn*}{Definition}
\newtheorem*{ques*}{Question}
\newtheorem*{conj*}{Conjecture}

\newtheorem*{prob*}{Problem}

\newtheorem{rem}[rem]{Remark}
\newtheorem{ex}[ex]{Example}
\newtheorem*{fact*}{Fact}
\newtheorem*{rem*}{Remark}
\newtheorem*{ex*}{Example}
\aliascntresetthe{prop}
\aliascntresetthe{cor}
\aliascntresetthe{lem}
\aliascntresetthe{claim}
\aliascntresetthe{defn}
\aliascntresetthe{ques}
\aliascntresetthe{conj}
\aliascntresetthe{fact}
\aliascntresetthe{rem}
\aliascntresetthe{ex}
\usepackage[pointedenum]{paralist}%
\usepackage{varioref}
\labelformat{equation}{\textnormal{(#1)}}
\labelformat{enumi}{\textnormal{(#1)}}

\setdefaultenum{(a)}{(i)}{(1)}{(A)}%

\usepackage{hyperref}

\def\textsectionN~{\textsection{}}

\renewcommand\phi{\varphi}
\renewcommand\epsilon{\varepsilon}
\renewcommand\leq{\leqslant}
\renewcommand\geq{\geqslant}

\makeatletter
\newcommand{\set}{%
  \@ifstar{\@setstar}{\@set}%
}%
\newcommand{\@setstar}[2]{\{\, #1 \mid #2 \,\}}
\newcommand{\@set}[1]{\{\, #1 \,\}}
\makeatother
\newcommand{\lin}[1]{\langle\, #1 \,\rangle}
\newcommand{\trans}[1][1]{\raisebox{#1ex}{\scriptsize\kern0.1em$t$\kern-0.1em}}

\newcommand{\PP}{\mathbb{P}}
\newcommand{\cP}{{\PP_{\!\! *}}}
\newcommand{\TT}{\mathbb{T}}
\newcommand{\PN}{\PP^N}
\newcommand{\Pv}[1][N]{(\PP^{#1})\spcheck}

\newcommand{\sO}{\mathscr{O}}

\newcommand{\X}{\mathcal{X}}
\newcommand{\Y}{\mathcal{Y}}

\newcommand{\bpi}{\bar\pi}
\newcommand{\Mm}{M^{-}}
\newcommand{\mpl}{m^{+}}
\newcommand{\mm}{m^{-}}
\newcommand{\Qp}{Q_{+}}
\newcommand{\Sp}{S_{+}}

\newcommand{\Gr}{\mathbb{G}}
\newcommand{\Gm}{\Gr(m,\PN)}
\newcommand{\Gmz}{\Gr(m_0,\PN)}
\newcommand{\Gmp}{\Gr(\mpl,\PN)}
\newcommand{\GM}{\Gr(M,\PN)}
\newcommand{\GMm}{\Gr(\Mm,\PN)}

\newcommand{\Go}{\mathbb{G}\spcirc}
\newcommand{\Gom}{\Go_{m}}
\newcommand{\Gomp}{\Go_{\mpl}}
\newcommand{\Xo}{\X\spcirc}
\newcommand\xo{x_o}

\newcommand{\sQ}{\mathscr{Q}}
\newcommand{\sS}{\mathscr{S}}
\newcommand{\sHom}{\mathscr{H}om}
\newcommand\spcirc{^\circ}

\newcommand\RNi[1][i]{0 \leq #1 \leq m}
\newcommand\RNj{m+1 \leq j \leq N}
\newcommand\RNij{\RNi,\RNj}
\newcommand\RNe{1 \leq e \leq \dim(\X)}

\newcommand\RNmu[1][\mu]{\mpl+1 \leq #1 \leq N}
\newcommand\RNnu{m+1 \leq \nu \leq \mpl}
\newcommand\RNlambda{0 \leq \lambda \leq \mpl}

\newcommand\Step[1]{\par\emph{Step #1.}}

\newcommand{\textgene}[1]{\ \ \text{#1}\ \,}

\newcommand{\textand}{\textgene{and}}

\DeclareMathOperator{\rk}{rk}%
\DeclareMathOperator{\codim}{codim}%
\DeclareMathOperator{\Hom}{Hom}%
\DeclareMathOperator{\im}{im}%
\DeclareMathOperator{\Proj}{Proj}
\DeclareMathOperator{\Spec}{Spec}
\DeclareMathOperator{\Tan}{Tan}

\title[Duality with expanding maps and shrinking maps\version]%
{Duality with \\ expanding maps and shrinking maps, and \\ its applications to Gauss maps\version}
\hypersetup{%
  pdftitle={Duality with expanding maps and shrinking maps\version},%
  pdfauthor={Katsuhisa FURUKAWA}%
}
\email{katu@toki.waseda.jp}
\author[K.~Furukawa]{Katsuhisa~FURUKAWA}
\urladdr{\url{http://www.aoni.waseda.jp/katu/index.html}}
\address{
  Department of Mathematics,
  School of Fundamental Science and Engineering,
  Waseda~University,
  Ohkubo~3-4-1, Shinjuku, Tokyo, 169-8555, Japan
}
\date{April 8, 2013}

\begin{document}

\maketitle

\begin{abstract}
  We study {expanding maps} and {shrinking maps} of subvarieties of Grassmann varieties
  in arbitrary characteristic.
  The shrinking map
  was studied independently
  by Landsberg and Piontkowski in order to characterize Gauss images.
  To develop their method, we introduce
  the expanding map,
  which is a dual notion of the shrinking map
  and is a generalization of the Gauss map.
  Then we give a characterization of separable Gauss maps and their images,
  which yields results for the following topics:
  (1) Linearity of general fibers of separable Gauss maps;
  (2) Generalization of the characterization of Gauss images;
  (3) Duality on one-dimensional parameter spaces of linear subvarieties lying in developable varieties.
\end{abstract}

\section{Introduction}
\label{sec:introduction}

For a projective variety $X \subset \PN$
over an algebraically closed field of arbitrary characteristic,
the \emph{Gauss map} $\gamma = \gamma_X$ of $X$
is defined to be the rational map
$X \dashrightarrow \Gr(\dim X, \PN)$
which sends each smooth point $x$ to the embedded tangent space $\TT_xX$ at $x$ in $\PN$.
The \emph{shrinking map} of a subvariety of a Grassmann variety
was studied independently by Landsberg and Piontkowski
in order to characterize \emph{Gauss images} in characteristic zero,
around 1996 according to \cite[p. 93]{IL}
(see \cite[2.4.7]{FP} and \cite[Theorem~3.4.8]{IL} for details of their results).
To develop their method,
we introduce the \emph{expanding map} of a subvariety of a Grassmann variety,
which is a generalization of the Gauss map
and is a dual notion of the shrinking map
(see \autoref{sec:expa-map-subv} for precise definitions of these maps).
Then we have the main theorem, \autoref{thm:mainthm-2},
which is a characterization of separable Gauss maps and their images in arbitrary characteristic,
and which yields results for the following topics.

\subsection{Linearity of general fibers of separable Gauss maps}
\begin{thm}[= \autoref{cor-mainthm}]
  \label{mainthm}
  Let $\gamma$ be a separable Gauss map of a projective variety $X \subset \PN$.
  Then the closure of
  a general fiber of $\gamma$ is a linear subvariety of $\PN$
  \footnote{  In the first version of this paper (arXiv:1110.4841v1),
    we proved the linearity
    by using duality on subspaces of tangent spaces
    with linear projection techniques.
    After that, we have examined what is essential in the proof.
    In the second version, considering expanding maps and shrinking maps,
    we give a more evident proof of the result.
  }.
\end{thm}

According to a theorem of
Zak \cite[I, 2.8. Corollary]{Zak},
the Gauss map is finite if $X$ is smooth (and is not a linear subvariety of $\PN$).
Combining with \autoref{mainthm}, we have that,
if the projective variety $X \subset \PN$ is smooth and the Gauss map $\gamma$ is separable,
then $\gamma$ is in fact birational
(\autoref{thm:cor-Gamma-birat}).
Geometrically, the birationality of $\gamma$ means that
a general embedded tangent space is tangent to $X$
at a unique point.

In characteristic zero,
it was well known that
the closure $F \subset X$ of a general fiber of the Gauss map $\gamma$
is a linear subvariety of $\PN$ (Griffiths and Harris \cite[(2.10)]{GH}, Zak \cite[I, 2.3. Theorem (c)]{Zak}).
In positive characteristic, 
$\gamma$ can be \emph{inseparable},
and then $F$ can be \emph{non-linear}
(see \autoref{thm:ch-p-non-lin});
this leads us to a natural question:
Is $F$ a linear subvariety if $\gamma$ is \emph{separable}?
(Kaji asked, for example, in \cite[Question 2]{Kaji2003-2} \cite[Problem 3.11]{Kaji2009}.)
The curve case was classically known (see \autoref{thm:dimX-1-gamma-bir}).
Kleiman and Piene \cite[pp.~108--109]{KP} proved that,
if $X \subset \PN$ is \textit{reflexive},
then a general fiber of the Gauss map $\gamma$
is scheme-theoretically (an open subscheme of)
a linear subvariety of $\PN$.
In characteristic zero,
their result
gives a reasonable
proof of the linearity of $F$,
since 
every $X$ is reflexive.
In arbitrary characteristic,
in terms of reflexivity,
the linearity of a general fiber $F$ of a separable $\gamma$ follows
if $\codim_{\PN} (X) = 1$ or  $\dim X \leq 2$,
since separability of $\gamma$ implies reflexivity of $X$
if $\codim_{\PN} (X) = 1$
(due to the Monge-Segre-Wallace criterion \cite[(2.4)]{HK}, \cite[I-1(4)]{Kleiman1986}),
$\dim X = 1$
(Voloch \cite{Voloch}, Kaji \cite{Kaji1992}),
or $\dim X = 2$ (Fukasawa and Kaji \cite{FK2007}).
On the other hand, for $\dim X \geq 3$,
Kaji \cite{Kaji2003} and Fukasawa \cite{Fukasawa2006-3} \cite{Fukasawa2007}
showed that separability of $\gamma$ does \emph{not}
imply reflexivity of $X$ in general.
For any $X$,
by \autoref{mainthm},
we finally answer the question affirmatively.

\subsection{Generalization of the characterization of Gauss images}

We generalize
the characterization of Gauss images given by Landsberg and Piontkowski
to the arbitrary characteristic case, as follows:

\begin{thm}[= \autoref{thm:cor-sep-gauss-image}]
  Let $\sigma$ be the shrinking map from a closed subvariety $\Y \subset \GM$
  to $\GMm$ with integers $M, \Mm$ ($M \geq \Mm$),
  and let $U_{\GMm} \subset \GMm \times \PN$ be the universal family of $\GMm$.
  Then
  $\Y$ is the closure of a image of a separable Gauss map if and only if
  $\Mm = M-\dim \Y$ holds and the projection
  $\sigma^*U_{\GMm} \rightarrow \PN$ is separable and generically finite onto its image.
\end{thm}

Here
the \emph{generalized conormal morphism},
induced from a expanding map,
plays an essential role;
indeed, we give a generalization of the Monge-Segre-Wallace criterion to the morphism
(\autoref{thm:rk-ineq}).

\subsection{Duality on one-dimensional developable parameter spaces}
\label{sec:dual-one-dimens}

Later in the paper,
instead of the subvariety $X \subset \PN$,
we focus on $\X \subset \Gm$, a parameter space of $m$-planes lying in $X$,
and study \emph{developability} of $(\X,X)$
(see \autoref{thm:def-developable}).
It is classically known that, in characterize zero,
a projective variety having a one-parameter developable uniruling (by $m$-planes)
is obtained as a cone over an \emph{osculating scroll} of a curve
(\cite[2.2.8]{FP}, \cite[Theorem.~3.12.5]{IL};
the arbitrary characteristic case was investigated by Fukasawa~\cite{Fukasawa2005}).
Applying our main theorem,
we find duality on
one-dimensional developable parameter spaces via expanding maps and shrinking maps, in arbitrary characteristic,
as follows.
Here
$\gamma^{i} = \gamma^{i}_{\X}$ is defined inductively by
$\gamma^1 := \gamma, \gamma^{i} := {\gamma_{\gamma^{i-1}\X} \circ \gamma^{i-1}}$,
with
the closure $\gamma^i\X$ of the image of $\X$ under $\gamma^i$.
In a similar way, $\sigma^i$ is defined.

\begin{thm}[= \autoref{thm:curve-gamma-sigma}]\label{thm:curve-gamma-sigma:1}
  Let $\X \subset \Gm$ and ${\X'} \subset \Gr(m', \PN)$ be projective curves.
  Then, for an integer $\epsilon \geq 0$, the following are equivalent:
  \begin{enumerate}
  \item
    ${\X'}$ is developable, 
    the map $\gamma^{\epsilon}= \gamma^{\epsilon}_{\X'}$ is separable and generically finite,
    and $\gamma^{\epsilon}{\X'} = \X$.
  \item\label{item:curve-gamma-sigma:b}
    ${\X}$ is developable,
    the map $\sigma^{\epsilon} = \sigma^{\epsilon}_{\X}$ is separable and generically finite,
    and $\sigma^{\epsilon}\X = {\X'}$.
  \end{enumerate}
  In this case, $m = m'+\epsilon$.
\end{thm}

As a corollary,
if $\sigma^m$ is separable and $X$ is not a cone,
then $C := \sigma^{m}\X$ is a projective curve in $\PN$
such that $\gamma^{m}$ is separable and $\X = \gamma^{m} C$;
in particular,
$X$
is equal to the osculating scroll of order $m$ of $C$
(\autoref{thm:one-para=oscu}).
On the other hand,
if $\gamma^2_{\X}$ is separable,
then an equality $T((TX)^*) = X^*$ in $\Pv$ holds
(\autoref{thm:one-dim-dual-tan};
cf. for osculating scrolls of curves,
this equality was deduced from Piene's work in characterization zero~\cite{Piene1983},
and 
was shown by Homma under some conditions on the characteristic~\cite{Homma}
(see \cite[Remark~4.3]{Homma})).
\vspace{1em}

This paper is organized as follows:
In \autoref{sec:expa-map-subv}
we fix our notation and give a local parametrization of a expanding map
$\gamma: \X \dashrightarrow \Gmp$ of
a subvariety $\X \subset \Gm$.
In addition, setting $\Y$ to be the closure of the image of $\X$,
we investigate properties of composition of the expanding map $\gamma$ of $\X$ and
the shrinking map $\sigma$ of $\Y$.
Then, in \autoref{sec:compo-expa-shr-maps} we prove the main theorem, \autoref{thm:mainthm-2}.
In \autoref{sec:developability} we regard $\X$ as a parameter space
of $m$-planes lying in $X \subset \PN$,
and study developability of $\X$ in terms of $\gamma$.

\section{Expanding maps of subvarieties of Grassmann varieties}
\label{sec:expa-map-subv}

In this section, we denote by $\gamma: \X \dashrightarrow \Gmp$
the expanding map of a subvariety $\X \subset \Gm$
with integer $m, \mpl$ ($m \leq \mpl$), which is defined as follows:

\begin{defn}
  Let $\sQ_{\Gm}$ and $\sS_{\Gm}$
  be the universal quotient bundle and subbundle of rank $m+1$ and $N-m$ on $\Gm$
  with the exact sequence
  $0 \rightarrow \sS_{\Gm} \rightarrow H^0(\PN, \sO(1)) \otimes \sO_{\Gm}\rightarrow \sQ_{\Gm} \rightarrow 0$.
  We set $\sQ_{\X} := \sQ_{\Gm}|_{\X}$
  and call this
  the universal quotient bundle on $\X$, and so on.
  We denote by $\X^{sm}$ the smooth locus of $\X$.
  A homomorphism $\phi$ is defined by the composition:
  \begin{equation*}
    \phi: \sS_{\X^{sm}} \rightarrow \sHom(\sHom(\sS_{\X^{sm}}, \sQ_{\X^{sm}}), \sQ_{\X^{sm}}) \rightarrow \sHom(T_{\X^{sm}}, \sQ_{\X^{sm}}),
  \end{equation*}
  where
  the first homomorphism
  is induced from the dual of
  $\sQ_{\X} \otimes \sQ_{\X}\spcheck \rightarrow \sO_{\X}$,
  and the second one
  is induced from
  $T_{\X^{sm}} \hookrightarrow T_{\Gm}|_{\X^{sm}} = \sHom(\sS_{\X^{sm}}, \sQ_{\X^{sm}})$.
  We can take an integer $\mpl = \mpl_{\gamma}$ with $m \leq \mpl \leq N$ such that
  a general point $x \in \X$ satisfies
  \begin{equation*}    \dim (\ker \phi \otimes k(x)) = N-\mpl.
  \end{equation*}
  Let $\Pv: = \Gr(N-1, \PN)$, the space of hyperplanes.
  Then $\ker \phi|_{\Xo}$ is a subbundle of
  $H^0(\PN, \sO(1)) \otimes \sO_{\Xo} \simeq H^0(\Pv, \sO(1))\spcheck \otimes \sO_{\Xo}$
  of rank $N-\mpl$ 
  for a certain open subset $\Xo \subset \X$.
  By the universality of the Grassmann variety, 
  under the identification $\Gr(N-\mpl-1, \Pv) \simeq \Gmp$,
  we have an induced morphism,
  \[
  \gamma = \gamma_{\X/\Gm}: \Xo \rightarrow  \Gmp.
  \]
  We call $\gamma$ the \emph{expanding map} of $\X$.
  Here $\ker \phi|_{\Xo} \simeq \gamma|_{\Xo}^*(\sS_{\Gmp})$.
\end{defn}

\begin{rem}\label{thm:expanding-map-Gauss-map}
  Suppose that $m=0$ and $X \subset \PN = \Gr(0, \PN)$.
  Then $\gamma = \gamma_{X/\PN}$ coincides with the Gauss map $X \dashrightarrow \Gr(\dim(X), \PN)$;
  in other words, $\gamma(x) = \TT_xX$
  for each smooth point $x \in X$.
  The reason is as follows:
  In this setting, it follows that $\sS_{\PN} = \Omega^1_{\PN}(1)$ and $\sQ_{\PN} = \sO_{\PN}(1)$,
  and that $\phi$ is the homomorphism $\Omega^1_{\PN}(1)|_X \rightarrow \Omega^1_X(1)$.
  Therefore $\ker\phi|_{\X^{sm}} = N_{X/\PN}\spcheck(1)|_{\X^{sm}}$, which implies the assertion.
\end{rem}

The shrinking map $\sigma: \Y \dashrightarrow \GMm$
of a subvariety $\Y \subset \GM$ with integers $M, \Mm$ ($M \geq \Mm$) is defined similarly, as follows:

\begin{defn}\label{thm:def-shr-map}
  Let $\sQ_{\Y}$ and $\sS_{\Y}$ be the universal quotient bundle and subbundle
  of rank $M+1$ and $N-M$ on $\Y$.
  A homomorphism $\Phi$ is defined by the composition:
  \[
  \Phi:  \sQ_{\Y^{sm}}\spcheck
  \rightarrow  \sHom(\sHom (\sQ_{\Y^{sm}}\spcheck, \sS_{\Y^{sm}}\spcheck), \sS_{\Y^{sm}}\spcheck)
  \rightarrow \sHom(T_{\Y^{sm}}, \sS_{\Y^{sm}}\spcheck),
  \]
  where the second homomorphism is induced from
  $T_{\Y^{sm}} \hookrightarrow T_{\GM}|_{\Y^{sm}} = \sHom(\sQ_{\Y^{sm}} \spcheck, \sS_{\Y^{sm}}\spcheck)$.
  We can take an integer $\Mm = \Mm_{\sigma}$
  with $-1 \leq \Mm \leq M$ such that a general point $y \in \Y$
  satisfies
  \[
  \dim (\ker \Phi \otimes k(y)) = \Mm + 1.
  \]
  Since $\ker \Phi|_{\Y\spcirc}$
  is a subbundle of $H^0(\PN, \sO(1))\spcheck \otimes \sO_{\Y\spcirc}$ of rank $\Mm+1$
  for a certain open subset $\Y\spcirc \subset \Y$,
  we have
  an induced morphism, called the \emph{shrinking map} of $\Y$,
  \[
  \sigma = \sigma_{\Y/\GM}: \Y\spcirc \rightarrow \GMm.
  \]
  Here we have $\ker \Phi|_{\Y\spcirc} = \sigma|_{\Y\spcirc}^*(\sQ_{\GMm}\spcheck)$.
\end{defn}

\begin{rem}\label{thm:dual-exp-shr-def}
  Let $\bar \X \subset \Gr(N-m-1,\Pv)$ be
  the subvariety corresponding to $\X$
  under the identification $\Gm \simeq \Gr(N-m-1,\Pv)$, and so on.
  Then $\gamma_{\X/\Gm}$ is identified with the shrinking map
  \[
  \sigma_{\bar\X/\Gr(N-m-1,\Pv)}: \bar\X \dashrightarrow \Gr(N-\mpl-1,\Pv)
  \]
  under $\Gmp \simeq \Gr(N-\mpl-1, \Pv)$.
  In a similar way,
  $\sigma_{\Y/\GM}$ is identified with the expanding map $\gamma_{\bar \Y/\Gr(N-M-1,\Pv)}$.
\end{rem}

Let $U_{\Gm} \subset \Gm \times \PN$ be the universal family of $\Gm$.
We denote by $U_{\X} := U_{\Gm}|_{\X} \subset \X \times \PN$
the universal family of $\X$, 
and by $\pi_{\X}: U_{\X} \rightarrow \PN$ the projection, and so on.
(Recall that, for each $x \in \X$,
the $m$-plane $x \subset \PN$ is equal to $\pi_{\X}(L_x)$
for the fiber $L_x$ of $U_{\X} \rightarrow \X$ at $x$.)

\begin{rem}\label{thm:basic-inclusion}
  A general point $x \in \X$ gives an inclusion
  $x \subset \gamma(x)$ of linear varieties in $\PN$, and
  a general point $y \in \Y$ gives an inclusion
  $\sigma(y) \subset y$ of linear varieties in $\PN$.
\end{rem}

\begin{lem}\label{thm:nondege-notcone}
  Let $\X, \Y, U_{\X}, U_{\Y}$ be as above. Then the following holds:
  \begin{enumerate}
  \item If $\mpl \leq N-1$ and the image of $\gamma$ is a point $L \in \Gmp$,
    then $\pi_{\X}(U_\X) \subset \PN$
    is contained in the $\mpl$-plane $L$.
  \item
    If $\Mm \geq 0$ and the image of $\sigma$ is a point $L \in \GMm$,
    then $\pi_{\Y}(U_{\Y})$ is a cone in $\PN$ such that the $\Mm$-plane $L$ is a vertex of the cone.
  \end{enumerate}
\end{lem}
\begin{proof}
  \begin{inparaenum}
  \item     For general $x \in \X$, we have $x \subset \gamma(x) = L$ as in \autoref{thm:basic-inclusion}.
    It follows that $\pi_{\X}(U_\X)$ is contained in the $\mpl$-plane $L$.

  \item         For general $y \in \Y$, we have $\sigma(y) = L \subset y$.
    We set $Y' := \pi_{\Y}(\Y) \subset \PN$.
    Then a general point $y' \in Y'$ is contained in some $M$-plane $y$,
    so that also $\lin{y', L}$
    is contained in $y$,
    where $\lin{y', L}$ is the linear subvariety of $\PN$ spanned by $y'$ and $L$.
    Hence $Y'$ is a cone with vertex $L$.
  \end{inparaenum}
\end{proof}

We denote by $\cP(A) := \Proj (\bigoplus \mathrm{Sym}^{d} A \spcheck)$
the projectivization of a locally free sheaf or a vector space $A$.
\begin{defn}\label{thm:defn-conormal}
  Let
  \[
  V_{\GM}:= \cP(\sS_{\GM}),
  \]
  which
  is contained in $ \GM \times \Pv = \cP(H^0(\PN, \sO(1)) \otimes \sO_{\GM})$
  and is regarded as the universal family of $\Gr(N-M-1, \Pv)$.
  We set $V_{\Y} := V_{\GM}|_{\Y}$ and set
  $\bpi = \bpi_{\Y}: V_{\Y} \rightarrow \Pv$ to be the projection.

  In the case where $\Y$
  is the closure of the image of $\X$ under the expanding map $\gamma$,
  the following commutative diagram is obtained:
  \begin{equation*}
    \xymatrix{      \gamma^*V_{\Y} \ar@{-->}[r] \ar@/^1pc/@<0.3em>[rr]^{\gamma^*\bpi} \ar[d]& V_{\Y} \ar[r]_{\bpi} \ar[d]& \Pv
      \\
      \X \ar@{-->}[r]_{\gamma} & \Y \makebox[0pt]{\,,}
    }  \end{equation*}
  where
  we call the projection $\gamma^*\bpi: \gamma^*V_{\Y} \rightarrow \Pv$
  the \emph{generalized conormal morphism},
  and where
  $\gamma^*V_{\Y} \subset \X \times \Pv$ is the closure of the pull-back
  $(\gamma|_{\Xo})^*V_{\Y}$.
  Note that we have $(\gamma|_{\Xo})^*V_{\Y} = \cP(\ker\phi|_{\Xo})$,
  because of $\gamma|_{\Xo}^*(\sS_{\Y}) \simeq \ker \phi|_{\Xo}$.
\end{defn}

\subsection{Standard open subset of the Grassmann variety}
\label{thm:desc-Go}
Let us denote by
\[
(Z^0: Z^1: \dots: Z^N)
\]
the homogeneous coordinates on $\PN$.
To fix our notation,
we will prepare a description of a standard open subset $\Gom \subset \Gm$
which is the set of $m$-planes not intersecting the $(N-m-1)$-plane
$(Z^0 = Z^1 = \dots = Z^{m} = 0)$.
Let us denote by
\[
Z_0, Z_1, \dots, Z_{N} \in H^0(\PN, \sO(1))\spcheck 
\]
the dual basis of $Z^0, Z^1, \dots, Z^{N} \in H^0(\PN, \sO(1))$, and so on.

\begin{inparaenum}[(A)]
\item \label{thm:desc-Go:Go}
  The sheaves $\sQ_{\Gom}$ and $\sS_{\Gom} \spcheck$
  are free on $\Gom$, and are equal to
  $Q \otimes \sO_{\Gom}$ and $S\spcheck \otimes \sO_{\Gom}$,
  for the vector spaces
  \[
  Q := \bigoplus_{\RNi} K \cdot \eta^i \textand
  S\spcheck := \bigoplus_{\RNj} K \cdot \zeta_j,
  \]
  where $K$ is the ground field,
  $\eta^i$ is the image of $Z^i$ under
  $H^0(\PN, \sO(1)) \otimes \sO \rightarrow \sQ_{\Gm}$,
  and $\zeta_j$ is the image of $Z_j$ under
  $H^0(\PN, \sO(1))\spcheck \otimes \sO \rightarrow \sS_{\Gm}\spcheck$.
  
  We have a standard isomorphism
  \begin{equation}\label{eq:HomQvSv}
    \Gom \simeq \Hom(Q\spcheck , S\spcheck):
    x \mapsto \sum_{\RNij}
    a^j_i \cdot \eta^i \otimes \zeta_j = (a^j_i)_{i,j},
  \end{equation}
  as follows.
  We take an element $x \in \Gom$.
  Under the surjection $H^0(\PN, \sO(1))\spcheck \rightarrow \sS_{\Gm} \spcheck \otimes k(x)$,
  for each $\RNi$,
  we have $Z_i \mapsto - \sum_{\RNj} a_i^j \cdot \zeta_j$
  with some $a_i^j = a_i^j(x) \in K$.
  This induces a linear map
  $Q\spcheck  \rightarrow S\spcheck: \eta_i \mapsto \sum a_i^j \cdot \zeta_j$,
  which is regarded as a tensor $\sum a_i^j \cdot \eta^i \otimes \zeta_j$
  under the identification $\Hom(Q\spcheck , S\spcheck) \simeq Q \otimes S\spcheck$.
  This gives the homomorphism~\ref{eq:HomQvSv}.

  In this setting, the linear map
  $\sQ_{\Gom} \spcheck \otimes k(x) \rightarrow H^0(\PN, \sO(1))\spcheck$
  is given by
  $\eta_i \mapsto Z_i  + \sum_j a_i^j \cdot Z_j$,
  and hence, for each $x \in \Gm$, the $m$-plane $x \subset \PN$ is spanned by
  the points of $\PN$ corresponding to the row vectors of the $(m+1)\times (N+1)$ matrix,
  \[
  \begin{bmatrix}
    1 &  &  & \bm 0 & a^{m+1}_{0}  & a^{m+2}_{0} & \cdots & a^{N}_{0} \\
    & 1 &  &        & a^{m+1}_{1}  & a^{m+2}_{1} & \cdots & a^{N}_{1} \\
    && \ddots      && \vdots       & \vdots      && \vdots \\
    \bm 0 &  &  & 1 & a^{m+1}_{m}  & a^{m+2}_{m} & \cdots & a^{N}_{m}
  \end{bmatrix}.
  \]

\item\label{thm:desc-Go:UGm}
  Let $U_{\Gm} := \cP(\sQ_{\Gm}\spcheck)$ in $\Gm \times \PN$,
  which is the universal family of $\Gm$.
  Then we have an identification
  \[
  U_{\Gom} \simeq \Gom \times \PP^{m} \simeq \Hom(Q\spcheck, S\spcheck) \times \PP^{m}.
  \]
  Regarding $(\eta^0: \dots: \eta^{m})$ as the homogeneous coordinates on $\PP^{m} = \cP(Q\spcheck)$,
  under the identification~\ref{eq:HomQvSv},
  we can parametrize the projection $U_{\Gom} \rightarrow \PN$ by sending 
  $((a_i^j)_{i,j}, (\eta^0: \dots: \eta^{m}))$ to the point
  \[
  (\eta^0: \dots: \eta^{m}: \sum_i \eta^i a_i^{m+1}: \dots: \sum_i \eta^i a_i^{N}) \in \PN.
  \]
  This is also expressed as
  \begin{equation}\label{eq:desc-Go:UGm}
    \sum_{\RNi} \eta^i Z_i  + \sum_{\RNij} \eta^i a_i^j \cdot Z_j \in \PN.
  \end{equation}

\item\label{thm:desc-Go:VGm}
  The $m$-plane $x\in \Gom$, which is expressed as
  $(a^j_i)_{i,j}$ under~\ref{eq:HomQvSv}, is also given by
  the set of points $(Z^0: Z^1: \dots : Z^{N}) \in \PN$ such that

  \[
  \begin{bmatrix}
    Z^0 & Z^1 & \cdots & Z^{N}
  \end{bmatrix}
  \begin{bmatrix}
    a^{m+1}_{0} & a^{m+2}_{0} & \cdots & a^{N}_{0} \\
    a^{m+1}_{1} & a^{m+2}_{1} & \cdots & a^{N}_{1} \\
    \vdots      & \vdots      && \vdots \\
    a^{m+1}_{m} & a^{m+2}_{m} & \cdots & a^{N}_{m} \\
    -1 &  &  & \bm 0  \\ 
    & -1 &  &  \\       
    && \ddots\\
    \bm 0 &  &  & -1 \\    
  \end{bmatrix}
  = 0.
  \]

  Let $V_{\Gm} := \cP(\sS_{\Gm})$ in $\Gm \times \Pv$,
  which is the universal family of $\Gr(N-m-1, \Pv)$.
  Then we have an identification
  \[
  V_{\Gom} \simeq \Gom \times \PP^{N-m-1} \simeq \Hom(Q\spcheck, S\spcheck) \times \PP^{N-m-1}.
  \]
  Regarding $(\zeta_{m+1}: \dots: \zeta_{N})$ as homogeneous coordinates on $\PP^{N-m-1} = \cP(S)$,
  we can parametrize $V_{\Gom} \rightarrow \Pv$
  by sending $((a_i^j)_{i,j}, (\zeta_{m+1}: \dots: \zeta_{N}))$ to
  the hyperplane defined by the homogeneous polynomial
  \begin{equation}\label{eq:hyp-pl-defpl}
    \sum_{\RNij} \zeta_j a^j_i \cdot Z^i + \sum_{\RNj} -\zeta_j \cdot Z^j.
  \end{equation}
\end{inparaenum}

\subsection{Parametrization of expanding maps}
\label{sec:param-expa-maps}

Let $\X \subset \Gm$ be a subvariety with $m \geq 0$.
We will give a local parametrization of the expanding map $\gamma: \X \dashrightarrow \Gmp$
around a general point $\xo \in \X$
in the following two steps.

\Step{1}
Changing the homogeneous coordinates $(Z^0: \dots: Z^{N})$ on $\PN$,
we can assume that $\xo \in \Gm$ and $\gamma(\xo) \in \Gmp$ are linear subvarieties of
dimensions $m$ and $\mpl$ such that
\begin{gather}
  \label{eq:xo}
  \xo = (Z^{m+1} = \dots = Z^{N} = 0), \\
  \label{eq:gxo}
  \gamma(\xo) = (Z^{\mpl+1} = \dots = Z^{N} = 0),
\end{gather}
in $\PN$.
As in \autoref{thm:desc-Go},
let us consider the open subset  $\Gom \subset \Gm$,
and
take a system of regular parameters
$z^1, \dots, z^{\dim (\X)}$ of the regular local ring $\sO_{\X,\xo}$.
Then, under the identification~\ref{eq:HomQvSv},
$\X \cap \Gom$ is parametrized around $\xo$
by
\[
\sum_{\RNij} f^{j}_{i} \cdot \eta^i \otimes \zeta_j = (f^j_i)_{i,j}
\]
with regular functions $f^{j}_{i}$'s.
From \ref{eq:xo}, we have $f^{j}_{i}(\xo) = 0$.
For a general point $x \in \X$ near $\xo$,
we identify $Q$ with $\sQ_{\X} \otimes k(x)$, and $S$ with $\sS_{\X} \otimes k(x)$.
Then the linear map $T_x\X \hookrightarrow T_{x}\Gm = \Hom(Q\spcheck, S\spcheck)$
is represented by
\begin{align}\label{eq:repre-TxX}
  \frac{\partial}{\partial z^e} &\mapsto \sum_{\RNij} f^j_{i, z^e}(x) \cdot \eta^i \otimes \zeta_j
  & (\RNe),
\end{align}
where $f^j_{i, z^e} := \partial f^j_{i}/ \partial z^e$.
Therefore
$\Hom(\Hom(S, Q), Q) \rightarrow \Hom(T_x\X, Q)$
is represented by
\begin{align*}
  \zeta^j \otimes \eta_i \otimes \eta^{i'} &\mapsto
  \sum_{\RNe} f^j_{i, z^e}(x) \cdot d z^e \otimes \eta^{i'}
  & (\RNi[i,i'], \RNj).
\end{align*}
Since $S \rightarrow \Hom(\Hom(S, Q), Q)$ is given by
$\zeta^j \mapsto \zeta^j \otimes \sum_i (\eta_i \otimes \eta^i)$,
it follows that $\phi_x: S \rightarrow \Hom(T_x\X, Q)$ is represented by
\begin{align}\label{eq:repre-phi_x-zeta}
  \phi_x(\zeta^j) &= \sum_{\RNi, \RNe} f^j_{i, z^e}(x) \cdot d z^e \otimes \eta^i
  &(\RNj).
\end{align}
Recall that $\mpl$ is the integer such that
$N-\mpl = \dim (\ker \phi_x)$, which implies that $\dim(\phi_x(S)) = \mpl-m$.
By \ref{eq:gxo},
we have $\phi_{\xo}(\zeta^{\mpl+1}) = \dots = \phi_{\xo}(\zeta^{N}) = 0$.
It follows that
$\phi_x(\zeta^{m+1}), \dots, \phi_x(\zeta^{\mpl})$ give a basis of the vector space $\phi_x(S)$,
and that
\begin{align}\label{eq:zeta-lin-dep}
  \phi_x(\zeta^{\mu}) &= \sum_{\RNnu} g^\mu_{\nu}(x) \cdot \phi_x(\zeta^{\nu})
  &(\RNmu)
\end{align}
with regular functions $g^{\mu}_{\nu}$'s. As a result, we have
\begin{align}\label{eq:f-mu-i-ze}
  f^{\mu}_{i,z^e} &= \sum_{\RNnu} g^{\mu}_{\nu} f^{\nu}_{i,z^e} & (\RNi, \RNe)
\end{align}
for $\RNmu$.
Since $\phi_{\xo}(\zeta^{\mu}) = 0$ for $\RNmu$, 
in this setting, it follows that $g^{\mu}_{\nu}(\xo) = 0$ and that
\begin{align}\label{eq:f-mu-i-z-o}
  f^{\mu}_{i, z^e}(\xo) &= 0 & (\RNi, \RNe)
\end{align}
for $\RNmu$.

\begin{lem}
  Let $\X \subset \Gm$ be of dimension $> 0$ with $m < N$.
  For the integer $\mpl$ given with $\gamma: \X \dashrightarrow \Gmp$,
  we have $\mpl > m$.
\end{lem}
\begin{proof}
  Assume $\mpl=m$. Then, as above, we have $\dim(\phi_x(S)) = 0$, which means that
  $\phi_x(\zeta^{j}) = 0$ for any $\RNj$. Then
  we have $f_{i,z^e}^j = 0$ for any $i,j,e$. This contradicts that $f_i^j$'s are regular functions parametrizing the embedding $\X \hookrightarrow \Gm$
  around the point $\xo$.
\end{proof}

\Step{2}
We set $\Y \subset \Gmp$ to be the closure of $\X$ under $\gamma$.
As in \autoref{thm:desc-Go}\ref{thm:desc-Go:VGm},
we set $V_{\Gm} := \cP(\sS_{\Gm})$ and consider the generalized conormal morphism
\[
\gamma^*\bpi|_{\Xo}: \cP(\ker\phi|_{\Xo}) \subset V_{\Gm} \rightarrow \Pv
\]
given in \autoref{thm:defn-conormal}.
Let $\ell_x$ be the fiber of $\cP(\ker\phi|_{\Xo}) \rightarrow \Xo$ at $x$,
and let $v \in \ell_x$ be a point.
Here $v$ is expressed as $((f^{j}_{i}(x))_{i,j}, (s_{m+1}: \dots: s_{N}))$.
Since $\sum_{\RNj} s_j \cdot \phi_x(\zeta^j) = \phi_x(\sum_{\RNj} s_j \zeta^j) = 0$,
the equality~\ref{eq:zeta-lin-dep} implies
\begin{align*}
  \sum_{\RNnu} s_{\nu} \cdot \phi_x(\zeta^{\nu})
  &= - \sum_{\RNmu} s_{\mu} \cdot \phi_x(\zeta^{\mu})
  \\
  &= \sum_{\RNnu, \RNmu} -s_{\mu} g^{\mu}_{\nu}(x) \cdot \phi_x(\zeta^{\nu}).
\end{align*}
Thus $s_{\nu} = \sum_{\RNmu} -s_{\mu} g^{\mu}_{\nu}$ for $\RNnu$.
Then it follows from \ref{eq:hyp-pl-defpl} that
each point $v \in \ell_x$ is sent to the hyperplane $\gamma^*\bpi(v) \in \Pv$
which is defined by the homogeneous polynomial
\begin{multline}\label{eq:hyp-pl-defpl-2}
  \sum_{\substack{\RNi\\ \RNnu}}
  \sum_{\RNmu} -s_{\mu} g^{\mu}_{\nu} f^{\nu}_i(x) \cdot Z^i
  +
  \sum_{\substack{\RNi\\ \RNmu}} s_{\mu} f^{\mu}_i(x) \cdot Z^i
  \\
  + \sum_{\substack{\RNnu\\ \RNmu}}  s_{\mu} g^{\mu}_{\nu}(x) \cdot Z^{\nu}
  + \sum_{\RNmu} - s_{\mu} \cdot Z^{\mu}.
\end{multline}
Note that, for the $\mpl$-plane $\gamma(x) \subset \PN$,
the image $\gamma^*\bpi(\ell_x)$ is equal to
$\gamma(x)^* \subset \Pv$, the set of hyperplanes containing $\gamma(x)$.

Now, the parametrization of the expanding map $\gamma: \X \dashrightarrow \Y$
is obtained, as follows:
Let $\sQ_{\Gmp}$ and $\sS_{\Gmp}$ be the universal quotient bundle and subbundle of rank $\mpl+1$ and $N-\mpl$
on $\Gmp$.
In a similar way to \autoref{thm:desc-Go}\ref{thm:desc-Go:Go},
we take a standard open subset $\Gomp$ as the set of $\mpl$-planes
not intersecting the $(N-\mpl-1)$-plane $(Z^0 = \dots = Z^{\mpl} = 0)$.
Then $\sQ_{\Gomp}$ and $\sS_{\Gomp}\spcheck$
are equal to $\Qp \otimes \sO_{\Gomp}$ and $\Sp\spcheck \otimes \sO_{\Gomp}$, for vector spaces
\[
\Qp = \bigoplus_{\RNlambda} K\cdot q^{\lambda} \textand
\Sp\spcheck = \bigoplus_{\RNmu} K\cdot s_{\mu},
\]
where $q^{\lambda}$ and $s_{\mu}$ correspond to $Z^{\lambda}$ and $Z_{\mu}$.

In this setting, by \ref{eq:hyp-pl-defpl} and \ref{eq:hyp-pl-defpl-2},
$\gamma(x) \in \Gomp = \Hom(\Qp\spcheck, \Sp\spcheck)$ is expressed as
\begin{equation}\label{eq:desc-expanding-map}
  \sum_{\substack{\RNi\\ \RNmu}}
  (f^{\mu}_i +   \sum_{\RNnu} - g^{\mu}_{\nu} f^{\nu}_i)(x) \cdot q^i \otimes s_{\mu}
  + \sum_{\substack{\RNnu\\ \RNmu}}  g^{\mu}_{\nu}(x) \cdot q^{\nu} \otimes s_{\mu}
\end{equation}
for a point $x \in \X$ near $\xo$.

\begin{ex}\label{thm:exa-expshr}
  \mbox{}
  \begin{inparaenum}[(i)]
  \item \label{thm:exa-expshr-0}
    We set $\X \subset \Gr(1,\PP^4)$ to be the surface which is the closure of the image of
    a morphism $\Spec K[z^1,z^2] \rightarrow \Go_1$ defined by
    \begin{equation}\label{eq:exa-para-X}
      (z^1,z^2) \mapsto
      (f_i^j)_{0 \leq i \leq 1, 2 \leq j \leq 4} = 
      \begin{bmatrix}
        f_0^2 & f_0^3 & f_0^4
        \\
        f_1^2 & f_1^3 & f_1^4
      \end{bmatrix}
      =
      \begin{bmatrix}
        z^1 & z^2 & 2z^1z^2
        \\
        0 & z^1 & (z^1)^2
      \end{bmatrix},
    \end{equation}
    where $(z^e)^a$ means that the $a$-th power of the parameter $z^e$, and so on.
    Assume that the characteristic is not $2$.

  \item \label{thm:exa-expshr-a}
    Let $\gamma_{\X}: \X \dashrightarrow \Gr(\mpl, \PP^4)$ be the expanding map of $\X$.
    Then $\mpl = 3$ and $\gamma_{\X}$ is expressed on $\Go_{3}$ by
    \begin{equation}\label{eq:exa-desc-exp}
      \trans
      \begin{bmatrix}
        -2z^1z^2 &  -(z^1)^2 & 2z^2 & 2z^1
      \end{bmatrix}.
    \end{equation}

    This is calculated as follows:
    First, we take the following $4 \times 3$ matrix $A$:
    \[
    A =
    \begin{bmatrix}
      (f_{i,z^1}^j)
      \\
      (f_{i,z^2}^j)
    \end{bmatrix}
    \ \text{ with }\ \,
    (f_{i,z^1}^j) =
    \begin{bmatrix}
      1 & 0 & 2z^2
      \\
      0 & 1 & 2z^1
    \end{bmatrix},\
    (f_{i,z^2}^j) =
    \begin{bmatrix}
      0 & 1 & 2z^1
      \\
      0 & 0 & 0
    \end{bmatrix}.
    \]
    By \ref{eq:repre-phi_x-zeta}, the rank of $\phi_x$ coincides with that of $A$,
    which is equal to $2$.
    Since $\dim(\phi_x(S)) = \mpl-m$, and since $m = 1$ in the setting,
    we have $\mpl = 3$.
    Moreover, the following equality holds in the matrix $A$:
    \[
    2z^2
    \begin{bmatrix}
      1\\0\\0\\0
    \end{bmatrix}
    + 2z^1
    \begin{bmatrix}
      0\\1\\1\\0
    \end{bmatrix}
    =
    \begin{bmatrix}
      2z^2\\2z^1\\2z^1\\0
    \end{bmatrix}.
    \]
    It follows from \ref{eq:f-mu-i-ze} that
    $g_2^4 = 2z^2, g_3^4 = 2z^1$.
    Now, \ref{eq:desc-expanding-map} yields the expression \ref{eq:exa-desc-exp} of $\gamma_{\X}$.
    More precisely, the calculation of \ref{eq:hyp-pl-defpl-2} is given by
    \[
    -g_2^4
    \begin{bmatrix}
      f_0^2
      \\
      f_1^2
      \\
      -1
      \\
      0
      \\
      0
    \end{bmatrix}
    -g_3^4
    \begin{bmatrix}
      f_0^3
      \\
      f_1^3
      \\
      0
      \\
      -1
      \\
      0
    \end{bmatrix}
    +
    \begin{bmatrix}
      f_0^4
      \\
      f_1^4
      \\
      0
      \\
      0
      \\
      -1
    \end{bmatrix}
    =
    \begin{bmatrix}
      -2z^1z^2
      \\
      -(z^1)^2
      \\
      2z^2
      \\
      2z^1
      \\
      -1
    \end{bmatrix}.
    \]

  \item\label{thm:exa-expshr-b}
    Let $\Y \subset \Gr(3, \PP^4)$ be the surface which is the closure of
    the image of $\X$ under $\gamma_{\X}$.
    Then the shrinking map $\sigma$ of $\Y$
    is a map from $\Y$ to $\Gr(1, \PP^4)$ and is indeed expressed on $\Go_1$ by
    \ref{eq:exa-para-X}.
    Hence the closure of the image of $\Y$ under $\sigma$ is equal to $\X$,
    and $\sigma \circ \gamma_{\X}$ coincides with the identity map
    on an open subset of $\X$.

    We note that
    one can calculate the expression of $\sigma$
    in a similar way to \ref{thm:exa-expshr-a},
    since $\sigma$ is identified with
    $\gamma_{\bar\Y/\Gr(0, \Pv[4])}$
    as in \autoref{thm:dual-exp-shr-def}.
    Here $\bar\Y \subset \Gr(0, \Pv[4]) = \Pv[4]$
    is the subvariety corresponding to $\Y \subset \Gr(3, \PP^4)$,
    and is parametrized by
    \[
    (1: -2z^1 : -2z^2 : (z^1)^2 : 2z^1z^2) = (-1: 2z^1 : 2z^2 : -(z^1)^2 : -2z^1z^2),
    \]
    where  the right hand side is given by
    transposing and reversing \ref{eq:exa-desc-exp}.
    
  \item\label{thm:exa-expshr-c}
    For later explanations,
    we consider a $3$-fold $X \subset \PP^4$ which is 
    the image of the projection $\pi_{\X}: U_{\X} \rightarrow \PP^4$,
    where
    $U_{\X} \subset \X \times \PP^4$ is the universal family of $\X$.
    As in \autoref{thm:desc-Go}\ref{thm:desc-Go:UGm},
    it follows from \ref{eq:exa-para-X} that
    $X$ is the closure of the image of a morphism
    $\Spec K[z^1,z^2, \eta^1] \rightarrow \PP^4$ defined by
    \[
    (1: \eta^1: z^1 : z^2 + \eta^1z^1 : 2z^1z^2 + \eta^1(z^1)^2).
    \]

  \item This $3$-fold $X \subset \PP^4$ is called a \emph{twisted plane} (see \cite[2.2.9]{FP}), and is defined by a homogeneous polynomial of degree $3$,
    \[
    (Z^0)^2Z^4+Z^1(Z^2)^2-2Z^0Z^2Z^3,
    \]
    where $(Z^0:Z^1: \dots: Z^4)$ is the homogeneous coordinates on $\PP^4$.

  \item\label{thm:exa-expshr-e}
    Let $\gamma_X: X \dashrightarrow \Gr(3, \PP^4)$ be the Gauss map of $X$.
    Then, in a similar way to \ref{thm:exa-expshr-a}, one can obtain the expression of $\gamma_X$,
    which indeed coincides with \ref{eq:exa-desc-exp}.
    In particular, the closure of the image of $X$ under $\gamma_X$
    is equal to $\Y$. 
  \end{inparaenum}
\end{ex}

\begin{ex}\label{thm:exa2-expshr}
  \mbox{}
  We set $\X \subset \Gr(1,\PP^5)$ to be the surface which is the closure of the image of
  a morphism $\Spec K[z^1,z^2] \rightarrow \Go_1$ defined by
  \begin{equation*}
    (z^1,z^2) \mapsto
    \begin{bmatrix}
      f_0^2 & f_0^3 & f_0^4 & f_0^5
      \\
      f_1^2 & f_1^3 & f_1^4 & f_1^5
    \end{bmatrix}
    =
    \begin{bmatrix}
      z^1 & z^2 & a\cdot (z^1)^{a-1}z^2 & h
      \\
      0 & z^1 & (z^1)^a & 0
    \end{bmatrix},
  \end{equation*}
  where $a$ is an integer greater than $1$, and where $h \in K[z^1]$.
  Assume that $a(a-1)$ is not divided by the characteristic.
  Then one can calculate several maps in a similar way to \autoref{thm:exa-expshr}.
  For instance,
  the expanding map $\gamma_{\X}$ is a map from $\X$ to $\Gr(3, \PP^5)$
  and is expressed on $\Go_{3}$ by
  \begin{equation*}
    \trans[2]
    \begin{bmatrix}
      -a(a-1) \cdot (z^1)^{a-1}z^2
      &
      -(a-1) \cdot (z^1)^a
      &
      a(a-1) \cdot (z^1)^{a-2}z^2
      &
      a\cdot (z^1)^{a-1}
      \\
      h-z^1h_{z^1}
      &
      0
      &
      h_{z^1}
      &
      0
    \end{bmatrix}.
  \end{equation*}
\end{ex}

\subsection{Composition of expanding maps and shrinking maps}
\label{sec:comp-expand-maps}

We set
$\X \subset \Gm$ to be a quasi-projective smooth variety,
and set $\Y \subset \Gmp$ to be the closure of 
the expanding map $\gamma : \X \dashrightarrow \Gmp$.
For this $\Y$, we will investigate the homomorphism $\Phi$
given in \autoref{thm:def-shr-map},
by considering
the pull-back of $\Phi$ via $\gamma$.
Now we have the following commutative diagram:
\begin{equation}\label{eq:Psi-factor}
  \xymatrix{    \gamma^*\sQ_{\Y}\spcheck
    \ar[r]
    \ar@/_1.7pc/@<0ex>[rrd]_{\Psi}
    \ar@/^1.6pc/@<0.5ex>[rr]^{\gamma^*\Phi}
    & 
    \sHom(\sHom(\gamma^*\sQ_{\Y}\spcheck, \gamma^*\sS_{\Y}\spcheck), \gamma^*\sS_{\Y}\spcheck)
    \ar[r]
    \ar[rd]_{- \circ d\gamma}
    & \sHom(\gamma^*T_\Y, \gamma^*\sS_{\Y}\spcheck) \ar[d] 
    \\
    && \sHom(T_\X, \gamma^*\sS_{\Y}\spcheck),
  }\end{equation}
where
$\Psi: \gamma^*\sQ_{\Y}\spcheck 
\xrightarrow{\gamma^*\Phi} \sHom(\gamma^*T_\Y, \gamma^*\sS_{\Y}\spcheck)
\rightarrow \sHom(T_\X, \gamma^*\sS_{\Y}\spcheck)$
is the composite homomorphism,
and
\[
d\gamma : T_\X \rightarrow \gamma^*T_{\Gmp} \simeq \sHom(\gamma^*\sQ_{\Y}\spcheck, \gamma^*\sS_{\Y}\spcheck)
\]
is the homomorphism of tangent bundles induced by $\gamma$.

We recall that a rational map $f: A \dashrightarrow B$ of varieties is said to be \textit{separable}
if the field extension $K(A)/K(f(A))$ is separably generated.
Here, the following three conditions are equivalent:
(i) $f$ is separable;
(ii) the linear map $d_x f : T_x A \rightarrow T_{f(x)} f(A)$ of Zariski tangent spaces
is surjective for general $x \in A$;
(iii) a general fiber of $f$ is reduced.
In characteristic zero, every rational map must be separable.
A rational map is said to be \textit{inseparable} if it is not separable.

\begin{rem}\label{thm:kernel-rem}
  If $\gamma$ is separable, then we have
  \begin{equation*}
    \ker \gamma^*\Phi|_{\Xo} = \ker \Psi|_{\Xo}
  \end{equation*}
  for a certain open subset $\Xo \subset \X$.
  This is because
  the vertical arrow in \ref{eq:Psi-factor},
  $\Hom(T_{\gamma(x)}\Y, \sS_{\Y}\spcheck \otimes \gamma(x))
  \rightarrow 
  \sHom(T_x\X, \sS_{\Y} \spcheck \otimes \gamma(x))$, 
  is injective
  at a general point $x \in \X$.
\end{rem}

Let $\xo \in \X$ be a general point.
In the setting of \autoref{sec:param-expa-maps}, for a point $x \in \X$ near $\xo$,
it follows from \ref{eq:desc-expanding-map} that
$d_x\gamma: T_x\X \rightarrow T_{\gamma(x)}\Gmp$ is represented by
\begin{equation}\label{eq:expanding-map-repre-dgamma}
  \frac{\partial}{\partial z^{e}}
  \mapsto
  \sum_{\substack{\RNi\\ \RNmu}}
  \sum_{\RNnu} - g^{\mu}_{\nu, z^e} f^{\nu}_i(x) \cdot q^i \otimes s_{\mu} 
  + \sum_{\substack{\RNnu\\ \RNmu}} g^{\mu}_{\nu, z^e}(x) \cdot q^{\nu} \otimes s_{\mu},
\end{equation}
where we apply the following equality obtained by \ref{eq:f-mu-i-ze}:
\[
(f^{\mu}_i +   \sum_{\RNnu} - g^{\mu}_{\nu} f^{\nu}_i)_{z^{e}}
= \sum_{\RNnu} - g^{\mu}_{\nu, z^{e}} f^{\nu}_i.
\]
Then, from the diagram~\ref{eq:Psi-factor},
the linear map
$\Psi_x: \Qp\spcheck \rightarrow \Hom(T_x\X, \Sp\spcheck)$
is represented by
\begin{equation}\label{eq:expanding-map-repre-phi-x}
  \begin{aligned}
    \Psi_x(q_i) &=
    \sum_{\substack{\RNe\\ \RNmu}}
    \sum_{\RNnu} - g^{\mu}_{\nu, z^e} f^{\nu}_i(x) \cdot dz^e \otimes s_{\mu}
    & (\RNi),
    \\
    \Psi_x(q_{\nu}) &=
    \sum_{\substack{\RNe\\ \RNmu}} g^{\mu}_{\nu, z^e}(x) \cdot dz^e \otimes s_{\mu}
    & (\RNnu).
  \end{aligned}
\end{equation}

\begin{lem}\label{thm:rk-d-gamma-and-psi}
  The ranks of the above linear maps are obtained as follows.
  \begin{enumerate}
  \item \label{thm:rk-d-gamma-and-psi:a}
    $\rk d_{x}\gamma$ is equal to the rank of the $\dim(\X) \times (N-\mpl)\cdot (\mpl-m)$ matrix
    \[
    \begin{bmatrix}
      g^{\mpl+1}_{m+1, z^{1}}(x) & \cdots & g^{\mu}_{\nu, z^{1}}(x)
      & \cdots & g^{N}_{\mpl, z^{1}}(x)
      \\
      \vdots && \vdots && \vdots
      \\
      g^{\mpl+1}_{m+1, z^{\dim(\X)}}(x) & \cdots & g^{\mu}_{\nu, z^{\dim(\X)}}(x)
      & \cdots & g^{N}_{\mpl, z^{\dim(\X)}}(x)
    \end{bmatrix}.
    \]

  \item \label{thm:rk-d-gamma-and-psi:b}
    $\rk \Psi_x$ is equal to the rank of the $(\mpl-m) \times (N-\mpl)\cdot \dim(\X)$ matrix
    \[
    \begin{bmatrix}
      g^{\mpl+1}_{m+1, z^{1}}(x) & \cdots & g^{\mu}_{m+1, z^{e}}(x) & \cdots & g^{N}_{m+1, z^{\dim(\X)}}(x)
      \\
      \vdots && \vdots && \vdots
      \\
      g^{\mpl+1}_{\mpl, z^{1}}(x) & \cdots & g^{\mu}_{\mpl, z^{e}}(x) & \cdots & g^{N}_{\mpl, z^{\dim(\X)}}(x)
    \end{bmatrix}.
    \]
  \end{enumerate}
\end{lem}
\begin{proof}
  Considering the matrix description of \ref{eq:expanding-map-repre-dgamma},
  we find that, each column vector not belonging to the matrix of \ref{thm:rk-d-gamma-and-psi:a}
  is expressed as
  \[
  \begin{bmatrix}
    \sum_{\RNnu} -g^{\mu}_{\nu, z^1} f^{\nu}_i(x)\\
    \vdots\\
    \sum_{\RNnu} -g^{\mu}_{\nu, z^{\dim(X)}} f^{\nu}_i(x)
  \end{bmatrix}.
  \]
  This vector is linearly dependent on column vectors of
  the matrix of \ref{thm:rk-d-gamma-and-psi:a};
  hence the assertion of \ref{thm:rk-d-gamma-and-psi:a} follows.
  In the same way, considering the matrix description of
  \ref{eq:expanding-map-repre-phi-x}, we have the assertion of
  \ref{thm:rk-d-gamma-and-psi:b}.
\end{proof}

\begin{prop}\label{thm:QR-subset-psi}
  $\sQ_{\Xo} \spcheck \subset \ker \Psi|_{\Xo}$
  for a certain open subset $\Xo \subset \X$
\end{prop}
\begin{proof} It is sufficient to show that
  $\sQ_{\X}\spcheck \otimes k(\xo) \subset \ker \Psi_{\xo}$ for a general point $\xo \in \X$.
  In the setting of \autoref{sec:param-expa-maps},
  since $f^j_i(\xo) = 0$, it follows from \ref{eq:expanding-map-repre-phi-x} that
  \[
  \Psi_{\xo}(q_0) = \cdots = \Psi_{\xo}(q_{m}) = 0.
  \]
  This implies that
  $\sQ_{\X}\spcheck \otimes k(\xo) \subset \ker \Psi_{\xo}$
  in $H^0(\PN, \sO(1))$.
\end{proof}

Let $\sigma = \sigma_{\Y/\Gmp}: \Y \dashrightarrow \Gmz$ be the shrinking map of $\Y \subset \Gmp$,
where we set
\[
m_0 := (\mpl)^-.
\]

\begin{cor}\label{thm:QR-subset-psi:cor}
  Assume that $\gamma$ is separable and
  assume that $\Psi_x$ is of rank $\mpl-m$ for general $x \in \X$.
  Then $m_0 = m$ and $\sQ|_{\Xo}\spcheck = \ker \gamma^*\Phi|_{\Xo}$ for a certain open subset $\Xo$;
  hence $\sigma \circ \gamma|_{\Xo}$ is an identity map of $\Xo \subset \X$.
\end{cor}
\begin{proof}
  From \autoref{thm:QR-subset-psi}, we have $\sQ|_{\Xo}\spcheck \subset \ker \Psi|_{\Xo}$.
  Since $(\mpl+1) - \rk \Psi_x = (\mpl+1) - (\mpl-m) = m+1$,
  we have $\sQ|_{\Xo}\spcheck = \ker \Psi|_{\Xo}$.
  It follows from \autoref{thm:kernel-rem}
  that
  \[
  \sQ|_{\Xo}\spcheck = \ker \gamma^*\Phi|_{\Xo}.
  \]
  We recall that, by universality,
  the morphism $\sigma \circ \gamma: \Xo \rightarrow \Gmz$
  is induced from
  $\ker \gamma^*\Phi|_{\Xo} \subset H^0(\PN, \sO(1))\spcheck \otimes \sO_{\Xo}$. Therefore
  $\sigma \circ \gamma$ coincides with the original embedding $\X \hookrightarrow \Gm$.
\end{proof}

For the universal family $U_{\X} \subset \X \times \PN$,
we define a rational map
\[
\tilde\gamma: U_\X \dashrightarrow \Y \times \PN,
\]
by sending $(x, x') \in U_{\X}$ with $x \in \X$ and $x' \in \PN$ to $(\gamma(x), x') \in \Y \times \PN$.
Let $\sigma^*U_{\Gmz} \subset \Y \times \PN$ be the closure of the pull-back of $U_{\Gmz}$ under $\sigma$.

\begin{cor}\label{thm:im-tgamma-sigma-U0}
  Assume that $\gamma$ is separable.
  Then the image of $U_\X$ under $\tilde\gamma$ is contained in $\sigma^*U_{\Gmz}$,
  and hence we have the following inclusion of linear varieties of $\PN$:
  \[
  x \subset \sigma\circ\gamma (x) \subset \gamma(x).
  \]
\end{cor}
\begin{proof}
  As in \autoref{thm:kernel-rem},
  $\ker \gamma^*\Phi|_{\Xo} = \ker \Psi|_{\Xo}$.
  Then \autoref{thm:QR-subset-psi}
  implies that $U_{\Xo}$ is contained in $\gamma^*\sigma^*U_{\Gmz} = \cP \ker(\gamma^*\Phi|_{\Xo})$.
  Hence we have $\tilde\gamma(U_{\Xo}) \subset \sigma^*U_{\Gmz}$.
\end{proof}

It is known that, in characteristic two, the Gauss map of every curve is inseparable.
This also happens to the expanding map, as follows.
\begin{lem}
  Assume that $\X \subset \Gm$ is a curve, and assume that
  $\gamma$ is generically finite.
  If the characteristic is two, then
  $\gamma$ is inseparable.
\end{lem}
\begin{proof}
  Let $\xo \in \X$ be a general point.
  In the setting of \autoref{sec:param-expa-maps},
  since $\X$ is a curve, it is locally parametrized around $\xo$ by one parameter $z$.
  It follows from
  \ref{eq:f-mu-i-ze} and $g^{\mu}_{\nu}(\xo) = 0$ that
  \[
  f^{\mu}_{i,z,z}(\xo) = \sum_{\RNnu} g^{\mu}_{\nu,z}f^{\nu}_{i,z}(\xo).
  \]
  In characteristic two, we have $f^{\mu}_{i,z,z}(\xo) = 0$.
  Since the above formula vanishes for each $\mu$ and $i$,
  it follows from \ref{eq:repre-phi_x-zeta} that 
  $\sum_{\RNnu} g^\mu_{\nu,z}(\xo) \cdot \phi_{\xo}(\zeta^{\nu}) = 0$.
  Since $\phi_{\xo}(\zeta^{\nu})$'s are linearly independent, $g^\mu_{\nu,z}(\xo) = 0$.
  By \autoref{thm:rk-d-gamma-and-psi}, $\rk d_{\xo} \gamma = 0$, that is to say,
  $\gamma$ is inseparable.
\end{proof}

\section{Duality with expanding maps and shrinking maps}
\label{sec:compo-expa-shr-maps}

Let $\sigma$ be the shrinking map
from a subvariety $\Y \subset \GM$ to $\Gmz$
with integers $M, m_0$ ($M \geq m_0$), as in \autoref{thm:def-shr-map} (we set $m_0:= \Mm$).
Let $\X_0 \subset \Gmz$ be the closure of the image of the map $\sigma$,
and let $\pi_0$ be the projection from the universal family
$U_{\X_0} \subset \X_0 \times \PN$ of $\X_0$
to $\PN$.
We set
\[
\sigma^*U_{\X_0} \subset \Y \times \PN
\]
to be the closure of
the pull-back of $U_{\X_0}$ under $\sigma$,
and set $\sigma^*\pi_0$ to be the projection from $\sigma^*U_{\X_0}$
to $\PN$.
Note that these constructions of $\sigma$ and $\sigma^*\pi_0$ depend only on $\Y$.
For a subvariety $X \subset \PN$,
we set
\[
\Gamma(X) := \overline{\set*{(\TT_xX, x)\in \GM \times \PN}{x \in X^{sm}}},
\]
the incidence variety
of embedded tangent spaces and their contact points,
where $X^{sm}$ is the smooth locus of $X$.

\begin{thm}[Main theorem]
  \label{thm:mainthm-2}

  Let $N, M$ be integers with $0 < M < N$,
  let $X \subset \PN$ be an $M$-dimension closed subvariety, and let
  $\Y \subset \GM$ be a closed subvariety.
  We set $\sigma$ as above, and so on.
  Then the following are equivalent:

  \begin{enumerate}
  \item
    The Gauss map
    $\gamma = \gamma_X: X \dashrightarrow \GM$ is separable, 
    and the closure of its image
    is equal to $\Y$.

  \item $\Gamma(X) = \sigma^*U_{\X_0}$ in $\GM \times \PN$.

  \item $\sigma^*\pi_0: \sigma^*U_{\X_0} \rightarrow \PN$
    is separable and generically finite,
    and its image is equal to $X$
    (in particular, the image is of dimension $M$.)
    
  \item[\textnormal{(c')}] $\sigma^*\pi_0: \sigma^*U_{\X_0} \rightarrow \PN$
    is separable
    and its image is equal to $X$,
    and $\sigma$ is separable.

  \end{enumerate}
\end{thm}

\begin{cor}\label{thm:mainthm-2-cor}
  Assume that one of the conditions (a-c') holds.
  Then $m_0 = M - \dim (\Y)$,
  and the diagram
  \begin{equation}\label{eq:maximal-developable-ruling}
    \begin{split}
      \xymatrix{        U_{\X_0} \ar[r]^{\pi_0} \ar[d] & X \ar@{-->}[d]^{\gamma_X}
        \\
        \X_0 \ar@{-->}[r]_{\gamma_{\X_0}} & \Y
      }    \end{split}
  \end{equation}
  is commutative,
  where ${\gamma_{\X_0}}$ is
  the expanding map of $\X_0$ and
  is indeed a birational map whose inverse is the shrinking map $\sigma$,
  and where $\sigma (y) \in \X_0$ corresponds to
  the closure of the fiber $\gamma_X^{-1}(y) \subset \PN$ for a general point $y \in \Y$.
\end{cor}

In this setting,
we will call $(\X_0, X)$ the \emph{maximal} developable parameter space (see \autoref{thm:def-X_0}).

\begin{ex}\label{thm:exa-expshr:2}
  Let $X \subset \PP^4$ be the $3$-fold given
  in \ref{thm:exa-expshr-c} of \autoref{thm:exa-expshr}.
  Here, the above $\X_0$ is obtained as
  the surface $\X \subset \Gr(1, \PP^4)$ in \ref{thm:exa-expshr-0},
  which is equal to the image of $\Y$ under $\sigma$
  as in \ref{thm:exa-expshr-b}.
  By \ref{thm:exa-expshr-e}, we can directly verify that
  the diagram \ref{eq:maximal-developable-ruling} is commutative.
\end{ex}

In \autoref{sec:line-gener-fibers} we will show
the implication (a) $ \Rightarrow $ (b) of \autoref{thm:mainthm-2},
which leads to the linearity of general fibers of separable Gauss maps
(\autoref{cor-mainthm}).
In \autoref{sec:imag-separ-gauss-1} we will show
(c') $ \Rightarrow $ (a),
and complete the proof of \autoref{thm:mainthm-2}.
Here the implication  (c) or (c') $ \Rightarrow $ (a) gives a generalization of the characterization 
of Gauss images (\autoref{thm:cor-sep-gauss-image}).
We note that both implications (a) $ \Rightarrow $ (b) and (c') $ \Rightarrow $ (a)
will be discussed in the same framework,
given in \autoref{sec:expa-map-subv} (see \autoref{thm:c-to-a-dual:rem}).

\begin{rem}\label{thm:b-ac}
  (b) $ \Rightarrow $ (a) of \autoref{thm:mainthm-2} holds, as follows.
  Let $\tilde\gamma: X \dashrightarrow \Gamma(X)$
  be the rational map defined by $x \mapsto (\gamma(x), x)$.
  Since $\sigma^*U_{\X_0} \rightarrow \Y$ is separable
  and since $\tilde\gamma$ is birational,
  the composite map $\gamma$ is separable.
  In addition, (b) $ \Rightarrow $ (c) holds,
  since $\sigma^*\pi_0$ is identified with the birational projection
  $\Gamma(X) \rightarrow X$ under the assumption.
  On the other hand, (c) $ \Rightarrow $ (c') holds, since,
  if $\sigma^*\pi_0: \sigma^*U_{\X_0} \rightarrow \PN$ is separable and generically finite onto its image,
  then $\sigma^*U_{\X_0} \dashrightarrow U_{\X_0}$ is separable,
  and then so is $\sigma$.
\end{rem}

\subsection{Separable Gauss maps and shrinking maps}
\label{sec:line-gener-fibers}

In this subsection, we consider
the Gauss map $X \dashrightarrow \Gr(\dim(X), \PN)$ of
a quasi-projective smooth subvariety $X \subset  \PN$,
which coincides with the expanding map of $X$,
as in \autoref{thm:expanding-map-Gauss-map}.
We denote by $\gamma$ the map,
and by $\Y \subset \Gr(\dim(X), \PN)$ the closure of the image of $\gamma$.
In the setting of \autoref{sec:comp-expand-maps} with $m = 0$,
we have a natural homomorphism $\xi: \gamma^*\sQ_{\Y}\spcheck \rightarrow T_X$
in the following commutative diagram with exact rows and columns:
\begin{equation*}
  \xymatrix{    & 0 \ar[d] & 0 \ar[d]
    \\
    & \sO_X(-1) \ar[d] \ar@{=}[r]& \sO_X(-1) \ar[d]
    \\
    0 \ar[r] & \gamma^*\sQ_{\Y}\spcheck \ar[r] \ar[d]_{\xi}
    &  H^0(\PN, \sO(1))\spcheck \otimes \sO_X \ar[r] \ar[d] & \gamma^*\sS_{\Y}\spcheck \ar[r] \ar@{=}[d] & 0
    \\
    0 \ar[r] & T_X(-1) \ar[r]\ar[d]
    & T_{\PN}(-1)|_X \ar[r] \ar[d] & N_{X/\PN}(-1) \ar[r] & 0 \makebox[0pt]{\,,}
    \\
    &0 & 0
  }\end{equation*}
where the middle column sequence is induced from the Euler sequence of $\PN$.
We note that the diagram yields $\ker(\xi) = \sO_X(-1)$.

Recall that $\Psi$ is a homomorphism given in the diagram~\ref{eq:Psi-factor} in \autoref{sec:comp-expand-maps}.

\begin{prop}\label{thm:two-kernels}
  Assume that $m = 0$.
  Then we have an equality,
  \[
  \ker \Psi|_{X\spcirc} = \ker d\gamma(-1) \circ \xi|_{X\spcirc}
  \]
  for a certain open subset $X\spcirc \subset X$.
\end{prop}
\begin{proof}
  It is sufficient to prove
  \begin{equation*}
    \ker \Psi \otimes k(\xo) = \ker (d\gamma(-1) \circ \xi) \otimes k(\xo)
  \end{equation*}
  for a general point $\xo \in X$.
  In the setting of \autoref{sec:param-expa-maps},
  changing coordinates on $\PN$,
  we can assume $\xo = (1: 0 : \dots: 0)$ and $\gamma(\xo) = (Z^{\dim (X)+1} = \dots = Z^{N} = 0)$.
  Then, by taking $z^1, \dots, z^{\dim(X)}$ as $f_0^1, \dots, f_0^{\dim(X)}$,
  the original embedding $X \hookrightarrow \PN$ can be locally parametrized by
  $(1: z^1: \dots: z^{\dim(X)}: f_0^{\dim(X)+1}: \dots: f_0^{N})$.
  We find
  \[
  \phi_x(\zeta^1) = dz^1 \otimes \eta^0,\,
  \phi_x(\zeta^2) = dz^2 \otimes \eta^0,\,
  \cdots,\, \phi_x(\zeta^{\dim(X)}) = dz^{\dim(X)} \otimes \eta^0
  \]
  in \ref{eq:repre-phi_x-zeta}.
  Hence we have $g_{\nu}^{\mu} = f_{0, z^{\nu}}^{\mu}$ in \ref{eq:f-mu-i-ze}.
  As in \ref{eq:expanding-map-repre-dgamma},
  the linear map
  $d_{\xo}\gamma: T_{\xo}X \rightarrow T_{\gamma(\xo)}\Gr(\dim(X),\PN)$ is represented by
  \begin{align*}
    \frac{\partial}{\partial z^{e}}
    &\mapsto
    \sum_{\substack{1 \leq \nu \leq \dim(X)\\ \dim(X)+1 \leq \mu \leq N}}
    f_{0, z^{\nu}, z^{e}}^{\mu}(\xo) \cdot s_{\mu} \otimes q^{\nu}
    &    (\RNe).
  \end{align*}
  It follows from \ref{eq:expanding-map-repre-phi-x} that
  $\Psi_{\xo}: \Qp\spcheck \rightarrow \Hom(T_xX, \Sp\spcheck)$ is represented by
  \[
  \Psi_{\xo}(q_0) = 0,\;
  \Psi_{\xo}(q_{\nu}) = \sum_{\substack{1 \leq e \leq \dim(X)\\ \dim(X)+1 \leq \mu \leq N}}
  f_{0, z^{\nu}, z^{e}}^{\mu}(\xo) \cdot dz^e \otimes s_{\mu}
  \quad (1 \leq \nu \leq \dim(X)).
  \]
  Since $\xi_{\xo} : \Qp\spcheck \rightarrow T_{\xo}X$ is obtained by
  $\xi_{\xo}(q_0) = 0$ and $\xi_{\xo}(q_{\nu}) = \partial / \partial z^{\nu}$ with $1 \leq \nu \leq \dim(X)$,
  the linear maps $d_{\xo}\gamma \circ \xi_{\xo}$ and $\Psi_{\xo}$ can be identified;
  in particular, their kernels coincide.
\end{proof}

\begin{thm}\label{thm:mainthm-2a}
  The implication \textnormal{(a) $ \Rightarrow $ (b)}
  of \autoref{thm:mainthm-2} holds.
\end{thm}

\begin{proof}
  Let $X \subset \PN$ be a projective variety,
  let $\Y \subset \Gr(\dim(X),\PN)$ be the closure of the image of $X$ under $\gamma$,
  and let $\Y \dashrightarrow \Gmz$ be the shrinking map with $m_0 := \dim(X)^{-}$.
  We apply the previous argument to the quasi-projective variety $X^{sm}$.

  Assume that $\gamma$ is separable. 
  Then $\ker \gamma^*\Phi|_{X\spcirc} = \ker \Psi|_{X\spcirc}$ as in \autoref{thm:kernel-rem}.
  Therefore \autoref{thm:two-kernels} implies
  \[
  \ker \gamma^*\Phi |_{X\spcirc} = \ker (d\gamma(-1) \circ \xi) |_{X\spcirc},
  \]
  where the right hand side is of rank $\dim(X) - \dim (\Y) + 1$ because of the separability of $\gamma$.
  This implies $m_0 = \dim(X) - \dim (\Y)$.

  On the other hand, it follows from \autoref{thm:im-tgamma-sigma-U0} that
  $\Gamma(X) \subset \sigma^*U_{\X_0}$.
  In fact, $\Gamma(X)$ and $\sigma^*U_{\X_0}$ coincide, since both have the same dimension, $\dim(X)$.
\end{proof}

Now, we give the proof of \autoref{mainthm}; more precisely, we have:
\begin{cor}\label{cor-mainthm}
  Assume that the Gauss map $\gamma: X \dashrightarrow \Y$ is separable,
  and let $y \in \Y$ be a general point.
  Then the $m_0$-plane $\sigma(y) \subset \PN$ is contained in $X$.
  Moreover, the fiber $\gamma^{-1}(y) \subset X^{sm}$ coincides with $\sigma(y) \cap X^{sm}$,
  whose closure is equal to $\sigma(y)$.
\end{cor}

This implies that the Gauss map is separable
if and only if its general fiber is scheme-theoretically
(an open subscheme of) a linear subvariety of $\PN$.
The latter condition immediately implies the former one,
since a fiber is reduced if it is scheme-theoretically
a linear subvariety.

\begin{proof}[Proof of \autoref{cor-mainthm}]
  Let $y \in \Y$ be a general point,
  and denote by $F_y$ the fiber of $\sigma^*U_{\X_0} \rightarrow \Y$ at $y$.
  We recall that
  the image $\sigma^*\pi_0(F_y)$ is equal to the $m_0$-plane $\sigma(y) \subset \PN$.
  If $\gamma$ is separable, then \autoref{thm:mainthm-2a}
  implies $\Gamma(X) = \sigma^*U_{\X_0}$, and then the following commutative diagram holds:
  \[
  \xymatrix{    \Gamma(X) \ar[d] \ar@{=}[r] & \sigma^*U_{\X_0} \ar[d] \ar[dl]_{\sigma^*\pi_0}
    \\
    X \ar@{-->}[r]_{\gamma} & \Y \makebox[0pt]{.}
  }  \]
  In particular, the $m_0$-plane $\sigma(y) = \sigma^*\pi_0(F_y)$
  is contained in $X$. On the other hand,
  it follows that $(\sigma^*\pi_0)^{-1}(\gamma^{-1}(y)) = F_y \cap (\sigma^*\pi_0)^{-1}(X^{sm})$.
  Since $\sigma^*\pi_0$ is surjective,
  we have $\gamma^{-1}(y) = \sigma(y) \cap X^{sm}$,
  where the right hand side is an open dense subset of $\sigma(y) \subset X$.
\end{proof}

Combining \cite[I, 2.8. Corollary]{Zak} and \autoref{cor-mainthm},
we have:

\begin{cor}\label{thm:cor-Gamma-birat}
  If the projective variety $X \subset \PN$ is smooth (and is non-linear),
  then the separable Gauss map $\gamma$ is in fact birational onto its image.
\end{cor}

\begin{rem}\label{thm:ch-p-non-lin}
  In positive characteristic, the Gauss map $\gamma$ can be \emph{inseparable}
  (Wallace \cite[\textsection 7]{Wallace}),
  and then,
  in contrast to the characteristic zero case,
  a general fiber $F$ of $\gamma$ can be \emph{non-linear}.
  Several authors gave examples where $F$ of an inseparable $\gamma$
  is not a linear subvariety
  (Kaji \cite[Example 4.1]{Kaji1986} \cite{Kaji1989}, Rathmann \cite[Example~2.13]{Rathmann},
  Noma \cite{Noma2001},
  Fukasawa \cite{Fukasawa2005} \cite{Fukasawa2006}).
\end{rem}

\begin{rem}\label{thm:dimX-1-gamma-bir}
  If $\dim X = 1$ and the Gauss map $\gamma$ is separable, then $\gamma$ is birational.
  This fact was classically known 
  for plane curves in terms of \textit{dual curves}
  (for example, see 
  \cite[p. 310]{Kleiman1976}, \cite[\textsection 9.4]{Iitaka}),
  and was shown for any curve by Kaji \cite[Corollary 2.2]{Kaji1989}.
\end{rem}

\subsection{Generalized conormal morphisms}
\label{sec:imag-separ-gauss-1}

We denote by
$\gamma$ the expanding map from
a closed subvariety $\X \subset \Gm$ to $\Gmp$,
and by $\Y \subset \Gmp$ the closure of the image of $\X$ under $\gamma$.
We consider the generalized conormal morphism
$\gamma^*\bpi: \gamma^*V_{\Y} \rightarrow \Pv$
given in \autoref{thm:defn-conormal}.
Let $Y$ be the image of $\gamma^*V_{\Y}$ under $\gamma^*\bpi$,
which is
a subvariety of $\Gr(N-1, \PN) = \Pv$.

We denote by $\sigma_Y = \sigma_{Y/\Gr(N-1, \PN)}$ the shrinking map from $Y$ to $\Gr(N-1-\dim Y, \PN)$.
As in Remarks \ref{thm:expanding-map-Gauss-map} and \ref{thm:dual-exp-shr-def},
the map $\sigma_Y$
is identified with
the Gauss map $Y \dashrightarrow \Gr(\dim Y, \Pv)$ which sends $y \in Y$ to $\TT_yY \subset \Pv$.
Denoting by $A^* \subset \Pv$ the set of hyperplanes containing a linear subvariety $A \subset \PN$,
we have $\sigma_Y(y)^* = \TT_yY$.
\begin{thm}\label{thm:c-to-a-dual}
  Let $\X$ be as above.
  Assume that
  $\gamma^*\bpi$ is separable and
  its image $Y$ is of dimension $N-m-1$,
  and assume that $\gamma$ is separable.
  Then the shrinking map
  $\sigma_{Y}$ of $Y$
  is separable and
  the closure of its image is equal to $\X$.
\end{thm}

\begin{rem}\label{thm:c-to-a-dual:rem}
  By considering the dual of the above statement,
  it follows that \autoref{thm:c-to-a-dual} is equivalent to
  ``(c') $\Rightarrow$ (a)'' of \autoref{thm:mainthm-2}.
  This is because, in the setting of \autoref{thm:mainthm-2}, we have
  $\sigma_{\Y/\GM} = \gamma_{\bar \Y/ \Gr(N-M-1, \Pv)}$,
  $\gamma_{X/\PN} = \sigma_{\bar X/ \Gr(N-1, \Pv)}$, and so on (see \autoref{thm:dual-exp-shr-def}).
\end{rem}

The following result is essential for the proof of \autoref{thm:c-to-a-dual},
and is indeed a generalization of the Monge-Segre-Wallace criterion.
Here we recall that
$\Psi_x: \Qp\spcheck \rightarrow \Hom(T_x\X, \Sp\spcheck)$ is the linear map given in
\ref{eq:Psi-factor}, \ref{eq:expanding-map-repre-phi-x} in \autoref{sec:comp-expand-maps}.

\begin{prop}\label{thm:rk-ineq}
  Let $\X \subset \Gm$, let $v \in \gamma^*V_{\Y}$ be a general point,
  and let  $x \in \X$ be the image of $v$ under $\gamma^*V_{\Y} \rightarrow \X$.
  Then the following holds:
  \begin{enumerate}
  \item\label{item:rk-ineq:a}
    $\rk d_{v}\gamma^*\bpi - (N-\mpl-1) \leq \rk d_x\gamma$
    and
    $\rk d_{v}\gamma^*\bpi - (N-\mpl-1) \leq \rk \Psi_x$.
  \item\label{item:rk-ineq:rk0}

    If $\rk d_{v}\gamma^*\bpi = N-\mpl-1$, then $\rk d_x \gamma = 0$.

  \item\label{item:rk-ineq:b}
    If $\gamma^*\bpi$ is separable,
    then $\TT_{\gamma^*\bpi(v)} Y \subset x^*$ in $\Pv$.
  \end{enumerate}
\end{prop}
To prove \autoref{thm:rk-ineq},
we will describe the linear map $d_{v_o}\gamma^*\bpi$ for 
a general point $v_o \in \gamma^*V_{\Y}$, as follows.
Under the setting of \autoref{sec:param-expa-maps},
the morphism $\gamma^*\bpi$ is expressed as
\ref{eq:hyp-pl-defpl-2},
where we have $V_{\Gomp} \simeq \Gomp \times \PP^{N-\mpl-1}$
as in \autoref{thm:desc-Go}\ref{thm:desc-Go:VGm}.
Let $v_o = (\xo, s_o) \in \gamma^*V_{\Y}$ with $\xo \in \X$ and $s_o \in \PP^{N-\mpl-1}$.
Changing homogeneous coordinates on $\PN$, we can assume that
\[
\xo = (Z^{m+1} = \dots = Z^{N} = 0) \subset v_o' := \gamma^*\bpi(v_o) = (Z^{N} = 0) \, \text{ in } \PN.
\]
We regard $(s_{\mpl+1}: \dots: s_{N})$ and $(Z_0: \dots: Z_{N})$ as
homogeneous coordinates on $\PP^{N-\mpl-1} = \cP(\Sp)$ and $\Pv$.
Then, since $v_o' = (Z^{N}=0)$, we have
\[
s_o = (s_{\mpl+1} = \dots = s_{N-1} = 0) \in \PP^{N-\mpl-1}.
\]
For affine coordinates
$\bar s_{\mu} := s_{\mu}/s_{N}$
on $\set{s_{N} \neq 0} \subset \PP^{N-\mpl-1}$,
we regard $z^1, \dots, z^{\dim (\X)}, \bar s_{\mpl+1}, \dots, \bar s_{N-1}$
as a system of regular parameters of
$\sO_{\gamma^*V_{\Y}, v_o}$.
In addition, we set $\bar s_{N} := 1$, and
take affine coordinates $\bar Z_{\alpha} := Z_{\alpha}/Z_N$ on $\set{Z_N \neq 0} \subset \Pv$.

Now, for a general point $v = (x,s) \in \gamma^*V_{\Y}$ near $v_o$,
which is expressed as $((f^j_i)_{i,j}, (\bar s_{\mpl+1}, \cdots, \bar s_{N-1}))$,
it follows from
\ref{eq:hyp-pl-defpl-2}
that
the linear map $d_v \gamma^*\bpi: T_{v} \gamma^*V_{\Y} \rightarrow T_{v'}\Pv$ is represented by
\begin{equation}\label{eq:para-dgammapi-0}
  \begin{aligned}
    \frac{\partial}{\partial z^{e}}
    &\mapsto
    \sum_{\RNi}
    \sum_{\substack{\RNnu\\\RNmu}}
    - \bar s_{\mu} g^{\mu}_{\nu,z^e} f^{\nu}_i \cdot \frac{\partial}{\partial \bar Z_i}
    + \sum_{\RNnu} \sum_{\RNmu} \bar s_{\mu} g^{\mu}_{\nu, z^e} \cdot \frac{\partial}{\partial \bar Z_{\nu}},
    \\
    \frac{\partial}{\partial \bar s}_{\bar\mu}
    &\mapsto
    \sum_{\substack{\RNi}} (f^{\bar\mu}_i +
    \sum_{\RNnu} - g^{\bar\mu}_{\nu} f^{\nu}_i) \cdot \frac{\partial}{\partial \bar Z_i}
    + \sum_{\substack{\RNnu}} g^{\bar\mu}_{\nu} \cdot \frac{\partial}{\partial \bar Z_{\nu}}
    - \frac{\partial}{\partial \bar Z_{\bar\mu}},
  \end{aligned}
\end{equation}
for $\RNe$ and $\mpl+1 \leq \bar\mu \leq N-1$.
Here the $N-\mpl-1$ elements
$d_{v}\gamma^*\bpi(\partial / \partial \bar s_{\mpl+1}), \dots,
d_{v}\gamma^*\bpi(\partial / \partial \bar s_{N-1})$
are linearly independent,
since each of them has
${\partial}/{\partial \bar Z_{\bar\mu}}$ as its tail term.
Moreover, setting a $\dim(\X) \times (\mpl - m)$ matrix
\[
G_{d_{v}\gamma^*\bpi}
:=
\begin{bmatrix}
  \sum_{\mu} \bar s_{\mu} g^{\mu}_{m+1, z^1}(x)& \cdots &
  \sum_{\mu} \bar s_{\mu} g^{\mu}_{\mpl, z^1}(x)
  \\
  \vdots && \vdots
  \\
  \sum_{\mu} \bar s_{\mu} g^{\mu}_{m+1, z^{\dim(\X)}}(x) & \cdots &
  \sum_{\mu} \bar s_{\mu} g^{\mu}_{\mpl, z^{\dim(\X)}}(x)
\end{bmatrix},
\]
we have
\begin{equation*}\label{eq:rk-gconormal}
  \rk d_{v} \gamma^*\bpi = N-\mpl-1 + \rk G_{d_{v}\gamma^*\bpi}.
\end{equation*}

\begin{proof}[Proof of \autoref{thm:rk-ineq}]
  \begin{inparaenum}
  \item 
    From \autoref{thm:rk-d-gamma-and-psi}\ref{thm:rk-d-gamma-and-psi:a},
    it follows that $\rk G_{d_{v_o} \gamma^*\bpi} \leq \rk d_{\xo}\gamma$.
    In addition, from \autoref{thm:rk-d-gamma-and-psi}\ref{thm:rk-d-gamma-and-psi:b},
    considering the transpose of the matrix,
    we have $\rk G_{d_{v_o} \gamma^*\bpi} \leq \rk \Psi_{\xo}$.

  \item
    Assume that $\rk d_{\xo}\gamma > 0$. Then it follows from
    \autoref{thm:rk-d-gamma-and-psi}\ref{thm:rk-d-gamma-and-psi:a}
    that $g^{\mu}_{\nu, z^e}(\xo) \neq 0$ for some $\mu, \nu, e$.
    Hence
    $G_{d_{v} \gamma^*\bpi} \neq 0$ for some $v$
    with $v \mapsto \xo$ under $\gamma^*V_{\Y} \rightarrow \X$.
    This implies that $\rk d_{v}\gamma^*\bpi > N-\mpl-1$.

  \item 
    Since $f_i^j(\xo) = 0$, the description~\ref{eq:para-dgammapi-0}
    implies that
    $\im(d_{v_o}\gamma^*\bpi)$
    is contained in the vector subspace of $T_{v_o'}\Pv$ spanned by
    $\partial /\partial \bar Z_{m+1}, \cdots, \partial /\partial \bar Z_{N-1}$.
    If $\gamma^*\bpi$ is separable,
    then  $\TT_{v_o'} Y \subset (Z_0 = \dots = Z_{m} = 0)$ in $\Pv$,
    where the right hand side is equal to $\xo^*$.
  \end{inparaenum}
\end{proof}

Recall that $\sigma: \Y \dashrightarrow \Gr(m_0, \PN)$ is the shrinking map with
$m_0 := (\mpl)^{-}$.

\begin{cor}\label{thm:sigma-gamma-id}
  Assume that
  $\gamma^*\bpi$ is separable and
  its image is of dimension $N-m-1$,
  and assume that $\gamma$ is separable.
  Then $m_0 = m$
  and $\sigma \circ \gamma|_{\Xo}$ is an identity map of a certain open subset $\Xo \subset \X$.
\end{cor}
\begin{proof}
  Since $\rk d_{v} \gamma^*\bpi = N-m-1$, it follows from \autoref{thm:rk-ineq}\ref{item:rk-ineq:a}
  that $\mpl-m \leq \rk \Psi_x$ (the equality indeed holds,
  due to \autoref{thm:rk-d-gamma-and-psi}\ref{thm:rk-d-gamma-and-psi:b}).
  Then \autoref{thm:QR-subset-psi:cor} implies the result.
\end{proof}

\begin{proof}[Proof of \autoref{thm:c-to-a-dual}]
  \autoref{thm:sigma-gamma-id} implies that
  $m_0 = m$ and that $\sigma \circ \gamma|_{\Xo}$ is an identity map of $\Xo$.
  On the other hand, since $Y$ is of dimension $N-m-1$, in the statement of
  \autoref{thm:rk-ineq}\ref{item:rk-ineq:b}, we have
  $\TT_{\gamma^*\bpi(v)}Y = x^*$ in $\Pv$.
  Since
  $\sigma_{Y}$
  is identified with the Gauss map
  $Y \dashrightarrow \Gr(\dim Y, \Pv)$,
  it follows that $\sigma_Y(\gamma^*\bpi(v)) = x$ in $\PN$.
  Now the diagram

  \[
  \xymatrix{    \gamma^*V_{\Y} \ar@/^1pc/[rr]^{\gamma^*\bpi} \ar@{-->}[r] \ar[d] & V_{\Y} \ar[r]_{\bpi} \ar[d] & Y \ar@{-->}[d]^{\sigma_{Y}}
    \\
    \X \ar@{-->}[r]^{\gamma} \ar@/_1pc/[rr]_{\mathrm{id}} & \Y \ar@{-->}[r]^{\sigma} & \X
  }  \]
  is commutative. In particular $\sigma_{Y}$ is separable, since $\gamma^*V_{\Y} \rightarrow \X$ is.
\end{proof}

\begin{proof}[Proof of \autoref{thm:mainthm-2}]
  (b) $ \Rightarrow $ (c) $ \Rightarrow $ (c') follows from \autoref{thm:b-ac}.
  (a) $ \Rightarrow $ (b) follows from \autoref{thm:mainthm-2a}.
  (c') $ \Rightarrow $ (a) follows from \autoref{thm:c-to-a-dual} and \autoref{thm:c-to-a-dual:rem}.
\end{proof}
\begin{proof}[Proof of \autoref{thm:mainthm-2-cor}]
  $\gamma_{\X_0} \circ \sigma$ is identity map due to the dual statement of \autoref{thm:sigma-gamma-id}.
  As in \autoref{cor-mainthm}, for general $y$, $\sigma(y)$ corresponds
  to the closure of the fiber $\gamma_X^{-1}(y)$.
  Thus the diagram~\ref{eq:maximal-developable-ruling} is commutative.
\end{proof}

From the equivalence (c) $ \Leftrightarrow $ (a), we have:

\begin{cor}\label{thm:cor-sep-gauss-image}
  Let $\sigma$ be the shrinking map
  from a closed subvariety $\Y \subset \GM$ to $\GMm$.
  Then $\Y$ is the closure of a image of a separable Gauss map if and only if
  $\Mm = M-\dim \Y$ holds and
  $\sigma^*U_{\GMm} \rightarrow \PN$ is separable and generically finite onto its image.
\end{cor}

\begin{rem}
  In the case where $m=0$,
  \autoref{thm:rk-ineq}\ref{item:rk-ineq:b}
  gives the statement of the Monge-Segre-Wallace criterion \cite[(2.4)]{HK}, \cite[I-1(4)]{Kleiman1986}.
\end{rem}

\begin{rem}
  For $\X \subset \Gm$, in the diagram of \autoref{thm:defn-conormal},
  $\gamma^*\bpi$ can be inseparable even if $\gamma$ is separable.
  The reason is as follows:
  If $m = 0$,
  then $\gamma^*\bpi$ coincides with the conormal map $C(X) \rightarrow \Pv$ in the original sense.
  Then, as we mentioned in \autoref{sec:introduction},
  Kaji \cite{Kaji2003} and Fukasawa \cite{Fukasawa2006-3} \cite{Fukasawa2007}
  gave  examples of non-reflexive
  varieties (i.e., $\gamma^*\bpi$'s are inseparable) whose Gauss maps are birational.
  This implies the assertion.

  Considering the dual of the above statement,
  in the setting for \autoref{thm:mainthm-2},
  we find that
  $\sigma^*\pi_0$ can be inseparable even if $\sigma$ is separable;
  in other words, separability of $\sigma$ is not sufficient
  to give an equivalent condition for separability of the Gauss map of $X$.
\end{rem}

\section{Developable parameter spaces}
\label{sec:developability}

In this section,
we set $\pi = \pi_{\X}$ to be the projection
\[
\pi : U_{\X} := U_{\Gm}|_{\X}\rightarrow \PN
\]
for a closed subvariety $\X \subset \Gm$,
and set $X := \pi(U_{\X})$ in $\PN$.

\begin{defn}\label{thm:def-developable}
  We say that $(\X,X)$
  is \emph{developable}
  if $X = \pi(U_{\X})$ and,
  for general $x \in \X$,
  the embedded tangent space $\TT_{x'}X$ is the same
  for any smooth points $x' \in X$ lying in the $m$-plane $x \subset \PN$,
  i.e.,
  the Gauss map $\gamma_{X}$ of $X$ is constant on $x \cap X^{sm}$
  (cf. \cite[2.2.4]{FP}).
  We also say that $\X$ is \emph{developable} if
  $(\X, \pi(U_{\X}))$ is developable.
  The variety $X$ is said to be \emph{uniruled (resp. ruled)} by $m$-planes
  if $\pi$ is generically finite (resp. generically bijective).
\end{defn}

Note that,
in the case where $\gamma_{X}$ is separable,
there exists a developable parameter space $(\X,X)$ of $m$-planes with $m > 0$
if and only if the dimension of the image of $\gamma_{X}$ is less than $\dim X$;
this follows from existence of the \emph{maximal} developable parameter space (see \autoref{thm:def-X_0}).

\begin{ex}\label{thm:exa-expshr:3}
  We take $\X \subset \Gr(1, \PP^4)$ and $X \subset \PP^4$ to be the surface and $3$-fold
  given in \autoref{thm:exa-expshr} (see also \autoref{thm:exa-expshr:2}).
  Then $(\X, X)$ is developable due to \ref{thm:exa-expshr-e};
  indeed, it is maximal.
\end{ex}

\subsection{Expanding maps and developable parameter spaces}
\label{sec:expand-maps-devel}

\begin{prop}\label{thm:proj-UX-PN}
  Let $\gamma = \gamma_{\X}: \X \dashrightarrow \Gmp$ be the expanding map of $\X \subset \Gm$.
  We recall that
  $d_u\pi: T_{u} U_{\X} \rightarrow T_{u'}\PN$ is the linear map of Zariski tangent spaces
  at $u \in U_{\X}, u' = \gamma(u) \in \PN$.
  Then the following holds:
  \begin{enumerate}
  \item 
    $\rk d_{u}\pi \leq m+\dim(\X)$ and $\rk d_{u}\pi \leq \mpl$
    for general $u \in U_{\X}$.
  \item \label{thm:proj-UX-PN:R+1}
    If $\rk d_{u}\pi = m$ for general $u \in U_{\X}$,
    then $\X$ is a point.

  \item\label{thm:proj-UX-PN:b}
    Assume that
    $\pi$ is separable,
    and let $x \in \X$ be a general point.
    Then
    the $\mpl$-plane $\gamma(x) \subset \PN$ is spanned by
    $\dim(X)$-planes $\gamma_{X}(u')$ with smooth points $u' \in X$ lying in the $m$-plane $x$.
  \end{enumerate}
\end{prop}
To show \autoref{thm:proj-UX-PN}, we will first describe the linear map $d_{u}\pi$, as follows.
As in \autoref{thm:desc-Go}\ref{thm:desc-Go:UGm},
we have $U_{\Gom} \simeq \Gom \times \PP^{m}$.
Let $u_o = (\xo, \eta_o) \in U_{\X}$ be a general point with $\xo \in \X \cap \Gom$ and $\eta_o \in \PP^{m}$.
Changing homogeneous coordinates on $\PN$, we can assume that
$\xo = (Z^{m+1} = \dots = Z^{N} = 0) \subset \PN$ and $u_o' := \pi(u_o) = (1: 0: \cdots: 0) \in \PN$.
Then we have
\[
\eta_o = (\bar\eta^1 = \dots = \bar\eta^{m} = 0) \in \PP^{m}.
\]
We can also assume $\gamma(\xo) = (Z^{\mpl+1} = \dots = Z^{N} = 0)$.
From the expression~\ref{eq:desc-Go:UGm}, the projection  $\pi: U_{\X} \rightarrow \PN$ sends
an element $u = (x, \eta) \in U_{\X}$ near $u_o$, which is described as
$((f_i^j)_{i,j}, (\eta^0: \dots: \eta^{m}))$,
to the point
\[
\sum_{\RNi} \eta^i \cdot Z_i  + \sum_{\RNij} \eta^i f_i^j \cdot Z_j \in \PN.
\]
Let us take affine coordinates $\bar\eta^i := \eta^i/\eta^0$
on $\set{\eta^0 \neq 0} \subset \PP^{m}$, and $\bar Z^{\alpha} := Z^{\alpha}/Z^0$
on $\set{Z^0 \neq 0} \subset \PN$.
Then we regard $z^1, \dots, z^{\dim(\X)}, \bar\eta^1, \dots, \bar\eta^{m}$
as a system of regular parameters
of $\sO_{U_{\X}, u_o}$. We set $\bar\eta^0 := 1$.

Now, for general $u = (x, \eta) \in U_{\X}$ near $u_o$,
the linear map $d_u\pi: T_{u} U_{\X} \rightarrow T_{u'}\PN$ is represented by
\begin{equation}\label{eq:du-pi-TU-TPN}
  \begin{aligned}
    \frac{\partial}{\partial \bar\eta^{\bar\imath}} &\mapsto 
    \frac{\partial}{\partial \bar Z^{\bar\imath}}  + \sum_{\RNj}  f_{\bar\imath}^j \cdot \frac{\partial}{\partial \bar Z^j} & (1 \leq {\bar\imath} \leq m),
    \\
    \frac{\partial}{\partial z^{e}} &\mapsto 
    \sum_{\RNij} \bar\eta^i f_{i,z^e}^j \cdot \frac{\partial}{\partial \bar Z^j} & (\RNe).
  \end{aligned}
\end{equation}
Here the $m$ elements $d_u\pi(\partial/\partial \bar\eta^{1}), \dots, d_u\pi(\partial/\partial \bar\eta^{m})$ are linearly independent.
For a point $u \in U_{\X}$ near $u_o$ such that $u \mapsto x$ under
$U_{\X} \rightarrow \X$,
setting a $\dim(\X) \times (N-m)$ matrix
\begin{equation*}  F_{d_u\pi}:=
  \begin{bmatrix}
    \sum_i \bar\eta^i f_{i, z^{1}}^{m+1}(x) & \dots & \sum_i \bar\eta^i f_{i, z^{1}}^{N}(x)
    \\
    \vdots && \vdots
    \\
    \sum_i \bar\eta^i f_{i, z^{\dim(\X)}}^{m+1}(x) & \dots & \sum_i \bar\eta^i f_{i, z^{\dim(\X)}}^{N}(x)
  \end{bmatrix},
\end{equation*}
we have
\[
\rk d_{u}\pi = \rk F_{d_u\pi} +m.
\]

\begin{proof}[Proof of \autoref{thm:proj-UX-PN}]
  \begin{inparaenum}
  \item 
    From \ref{eq:repre-TxX} in \autoref{sec:param-expa-maps},
    we have $\rk F_{d_{u_o}\pi} \leq \dim(\X)$.
    From \ref{eq:repre-phi_x-zeta}, we have $\rk F_{d_{u_o}\pi} \leq \rk \phi_{\xo} = \mpl-m$.
    Thus the assertion follows.

  \item 
    If $\dim \X > 0$, then $f_{i,z^e}^j(\xo) \neq 0$ for some $i,j,e$.
    It follows that
    $F_{d_u\pi} \neq 0$ for some $u$ with $u \mapsto \xo$.
    This implies that
    $\rk d_u \pi > m$.

  \item
    Changing coordinates, we have that
    $\gamma(\xo) = (Z^{\mpl+1} = \dots = Z^{N} = 0)$ in $\PN$.
    Then the equality \ref{eq:f-mu-i-z-o} implies that,
    for each $u \in U_{\X}$ near $u_o$ with $u \mapsto \xo$, we have
    \[
    F_{d_u\pi}(x_o)=
    \begin{bmatrix}
      \begin{pmatrix}
        \sum_i \bar\eta^i f_{i, z^{1}}^{m+1}(\xo) & \dots & \sum_i \bar\eta^i f_{i, z^{1}}^{\mpl}(\xo)
        \\
        \vdots && \vdots
        \\
        \sum_i \bar\eta^i f_{i, z^{\dim(\X)}}^{m+1}(\xo) & \dots & \sum_i \bar\eta^i f_{i, z^{\dim(\X)}}^{\mpl}(\xo)
      \end{pmatrix}
      & \bm 0
    \end{bmatrix}.
    \]
    In addition, we recall that $f^j_i(\xo) = 0$.

    Now, we find an inclusion of linear varieties $\gamma_{X}(\pi(u)) \subset \gamma(\xo)$ in $\PN$,
    as follows:
    Considering the description~\ref{eq:du-pi-TU-TPN},
    we have that $\im(d_u\pi)$ is contained in the vector subspace of $T_{\pi(u)}\PN$
    spanned by ${{\partial}/{\partial \bar Z^{1}}}, \cdots, {{\partial}/{\partial \bar Z^{\mpl}}}$.
    Since $\pi$ is separable, $\gamma_X(\pi(u))$
    is contained in $\gamma(\xo) = (Z^{\mpl+1} = \dots = Z^{N} = 0)$.

    Suppose that there exists an $(\mpl-1)$-plane
    $L \subset \PN$ contained in the $\mpl$-plane $\gamma(\xo)$, such that
    $\gamma_{X}(u') \subset L$ holds for each smooth point $u' \in X$ lying in the $m$-plane $\xo$.
    Then we find a contradiction, as follows:
    Changing coordinates, we can assume
    \[
    L = (Z^{\mpl} = 0) \cap \gamma(\xo).
    \]
    Since $\pi$ is separable and since $\gamma_{X}(\pi(u)) \subset L$
    for each $u \in U_{\X}$ with $u \mapsto \xo$,
    considering the above matrix $F_{d_u\pi}$ and
    the description~\ref{eq:du-pi-TU-TPN} of $d_u\pi$,
    we have $f^{\mpl}_{i, z^e} (\xo) = 0$
    for each $i,e$. Then $\phi_x(\xi^{\mpl}) = 0$ due to \ref{eq:repre-phi_x-zeta}.
    This contradicts
    that a basis of the vector space $\phi_x(S)$ consists of $\phi_x(\xi^{m+1}), \cdots, \phi_x(\xi^{\mpl})$.
  \end{inparaenum}
\end{proof}

\begin{cor}
  In the setting of \autoref{thm:rk-ineq},
  if the maps $\gamma$ and $\bpi$ are separable,
  then we have $\TT_{\gamma^*\bpi(v)}Y \subset (\sigma \circ \gamma (x))^* \subset x^*$ in $\Pv$.
\end{cor}
\begin{proof}
  From \autoref{thm:im-tgamma-sigma-U0}, we have $x \subset \sigma \circ \gamma (x)$ in $\PN$.
  By applying the dual statement of \autoref{thm:proj-UX-PN}\ref{thm:proj-UX-PN:b} to $\bpi$ and $\gamma(x)$,
  the inclusion $\TT_{\gamma^*\bpi(v)}Y \subset (\sigma \circ \gamma (x))^*$ holds.
\end{proof}

We have the following criterion for developability (cf. \cite[2.2.4]{FP}),
where recall that $\mpl$ is an integer given with
the expanding map $\gamma_{\X}: \X \dashrightarrow \Gmp$.

\begin{cor}\label{thm:expand-develop}
  Assume that
  $\pi$ is separable.
  Then $\dim(X) = \mpl$ if and only if
  $(\X,X)$ is developable.
  In this case,
  the following commutative diagram holds:
  \begin{equation}\label{thm:expand-develop:diagram}
    \begin{split}
      \xymatrix{        U_{\X} \ar[r]^{\pi} \ar[d] & X \ar@{-->}[d]^{\gamma_{X}}
        \\
        \X \ar@{-->}[r]_{\gamma_\X} & \Gmp \makebox[0pt]{\,.}
      }    \end{split}
  \end{equation}
\end{cor}
\begin{proof}
  In \autoref{thm:proj-UX-PN}\ref{thm:proj-UX-PN:b},
  $\gamma_{\X}(x) = \gamma_X(u')$ holds if and only if
  the linear subvariety $\gamma_X(u') \subset \PN$
  is of dimension $\mpl$.
  Thus the assertion follows.
\end{proof}

In the case where $\pi$ is generically finite, $\dim (X) = \dim (\X)+m$; hence we also have:
\begin{cor}\label{thm:crit-developable}
  Assume that $\pi$ is separable and generically finite.
  Then we have $\dim(\X) = \mpl-m$ if and only if $\X$ is developable.
\end{cor}

\begin{ex}\label{thm:exa-expshr:4}
  In the setting of \autoref{thm:exa-expshr:3},
  we can also verify that $(\X, X)$ is developable by using
  \autoref{thm:crit-developable}
  (without calculation in \ref{thm:exa-expshr-e} of \autoref{thm:exa-expshr});
  this is because, we have $\mpl -m = 2$ in \ref{thm:exa-expshr-a}, which implies that
  the equality ``$\dim(\X) = \mpl-m$'' holds. In a similar way, one can show that the space $\X \subset \Gr(1, \PP^5)$ in \autoref{thm:exa2-expshr} is developable.
\end{ex}

\begin{defn}\label{thm:def-X_0}
  \mbox{}
  Let $X \subset \PN$ be a projective variety
  whose Gauss map $\gamma_{X}$ is separable.
  Then we set
  \[
  \X_0 \subset \Gmz
  \]
  to be the closure of the space which parametrizes
  (closures of) general fibers of $\gamma_X$,
  and call $(\X_0, X)$
  the \emph{maximal} developable parameter space.

  \begin{inparaenum}
  \item
    From \autoref{thm:mainthm-2-cor}, $\X_0$ can be obtained as
    the closure of the image of $X$ under the composite map
    $\sigma_{\Y} \circ \gamma_{X}$.
    In particular, the projection $\pi_0: U_{\X_0} \rightarrow X$
    is birational and the expanding map $\gamma_{\X_0}$ is birational.

  \item 
    For any developable $(\X,X)$ with $\X \subset \Gm$,
    there exists a dominant rational map $\X \dashrightarrow \X_0$
    through which $\gamma_\X: \X \dashrightarrow \Y$ factors.
    (This is because, for each $x \in \X$,
    we have an inclusion $x \cap X^{sm} \subset \gamma_X^{-1}(\gamma_{\X}(x))$ in $\PN$.
    Indeed, since $\gamma_{\X_0} \circ \sigma_\Y = id$,
    the map $\X \dashrightarrow \X_0$ is given by $\sigma_\Y\circ \gamma_\X$.)
  \end{inparaenum}
\end{defn}

\begin{rem}\label{thm:X-sep-ruled}
  Let $\X \subset \Gm$ be a subvariety such that $(\X,X)$ is developable.

  \begin{inparaenum}
  \item 
    Assume that $\pi$ is separable and generically finite,
    and assume that $\gamma_{\X}$ is generically finite.
    Then $\pi$ is indeed birational
    (i.e, $X$ is separably ruled by $m$-planes).
    The reason is as follows:
    From the diagram~\ref{thm:expand-develop:diagram},
    for general $x \in \X$,
    since $m$ is equal to the dimension of the fiber of
    $\gamma_X^{-1}(\gamma_{\X}(x))$,
    the $m$-plane $x \subset \PN$ is set-theoretically equal to an irreducible component of
    the closure of the fiber $\gamma_X^{-1}(\gamma_{\X}(x))$.
    This implies that $\pi$ is generically injective, and hence is birational.

  \item 
    Assume that $\pi$ and $\gamma_{\X}$ are separable and generically finite.
    Then $\X$ is equal to the parameter space $\X_0$
    given in
    \autoref{thm:def-X_0}.
    The reason is as follows:
    If $\gamma_{\X}$ is separable, then so is $\gamma_X$.
    It follows from \autoref{cor-mainthm} that
    the closure of the fiber
    $\gamma_X^{-1}(\gamma_{\X}(x))$ is irreducible,
    and hence is equal to the $m$-plane $x$.
    Thus $\X = \X_0$.

  \end{inparaenum}
\end{rem}

For example, in the following situation, the maximal developable parameter space for the dual variety of
$X \subset \PN$ can be obtained:

\begin{prop}\label{thm:ref-finite-proj}
  Let $\gamma_X: X \dashrightarrow \Y \subset \Gr(\dim(X),\PN)$ be the Gauss map,
  and let $Y := \bpi(V_{\Y})$ in $\Pv$, the dual variety of $X$.
  If $X$ is reflexive and $\bpi$ is generically finite,
  then $(\Y,Y)$ is the maximal developable parameter space
  with the birational projection
  $\bpi: V_{\Y} \rightarrow Y$, and then the following diagram is commutative:
  \[
  \xymatrix{    \gamma^*V_{\Y} \ar@{-->}[r] \ar[d] & V_{\Y} \ar[r]^{\bpi} \ar[d] & Y \ar@{-->}[d]^{\sigma_Y}
    \\
    X \ar@{-->}[r]^{\gamma_X} & \Y \ar@{-->}[r]^{\sigma_{\Y}} & \X_0,
  }  \]
  where note that
  the shrinking map $\sigma_Y: Y \dashrightarrow \X_0$ 
  is identified with
  the Gauss map
  $\gamma_{Y/\Pv}: Y \dashrightarrow \Gr(N-m_0-1, \Pv)$.
\end{prop}
\begin{proof}
  Since $X$ is reflexive, $\gamma^*\bpi$ is separable due to the Monge-Segre-Wallace criterion,
  and so is $\bpi$.
  Let $M := \dim X$.
  Since $\dim \Y = M-m_0 = (N-m_0-1)-(N-M-1)$,
  it follows from \autoref{thm:crit-developable} that $(\Y,Y)$ is developable
  and that the diagram is commutative.
  Since $\gamma_X$ is separable, $\sigma_{\Y}$ is birational (see \autoref{thm:mainthm-2-cor}).
  Hence, considering the dual statement of \autoref{thm:X-sep-ruled}, we have the assertion.
\end{proof}

\begin{rem}\label{thm:one-dim-ref}
  Suppose that $\Y$ is of dimension one. Then $\bpi$ is always separable and generically finite
  (see \autoref{thm:curve-sep-uniruling} below).
  In this case, $X$ is reflexive if and only if $\gamma_X$ is separable.
\end{rem}

\subsection{One-dimensional developable parameter space}
\label{sec:one-dimensional-base}

In this subsection, we assume that $\X \subset \Gm$ is a projective curve.
As above, we denote by $\pi = \pi_{\X}: U_{\X} \rightarrow \PN$ the projection,
and by $X := \pi(U_{\X})$ in $\PN$.
Here separability of $\pi$ always holds;
this is deduced from \cite{Kaji1992},
and can be also shown, as follows:
\begin{lem}\label{thm:curve-sep-uniruling}
  Let $\X$ be as above. Then
  $\pi$ is separable and generically finite.
\end{lem}
\begin{proof}
  Note that $U_{\X}$ is of dimension $m+1$.
  Since $\X$ is a curve,
  it follows from \autoref{thm:proj-UX-PN}\ref{thm:proj-UX-PN:R+1} that $\rk d_u \pi \geq m+1$.
  Thus $\rk d_u \pi = m+1$, which implies that
  $\pi$ is separable and generically finite.
\end{proof}

Let us consider the expanding map $\gamma: \X \dashrightarrow \Gr(\mpl_{\gamma}, \PN)$ 
and shrinking map $\sigma: \X \dashrightarrow \Gr(\mm_{\sigma}, \PN)$ of $\X$,
with integers $\mpl_{\gamma}$ and $\mm_{\sigma}$ ($\mm_{\sigma} < m < \mpl_{\gamma}$).

\begin{lem}\label{thm:curve-2m-lem}
  $\mpl_{\gamma}+\mm_{\sigma} = 2m$.
\end{lem}
\begin{proof}
  In the setting of \autoref{sec:param-expa-maps}, we consider the matrix
  \[
  F =
  \begin{bmatrix}
    f_{0,z}^{m+1} & \dots & f_{0,z}^{N}
    \\
    \vdots && \vdots
    \\
    f_{m,z}^{m+1} & \dots & f_{m,z}^{N}
  \end{bmatrix},
  \]
  where note that, since $\X$ is of dimension one,
  the system of parameters of $\sO_{\X,x}$ consists of one element $z$.
  Recalling the formula~\ref{eq:repre-phi_x-zeta},
  we have $\rk F = \dim(\phi_x(S)) = \mpl_{\gamma}-m$.
  In the same way, we have $\rk F = m-\mm_{\sigma}$.
  Thus the assertion follows.
\end{proof}

\begin{cor}\label{thm:curve-developability}
  The following are equivalent:
  \begin{enumerate}
  \item $\X$ is developable.
  \item $\mpl_{\gamma} = m+1$.
  \item $\mm_{\sigma} = m-1$.
  \end{enumerate}
\end{cor}
\begin{proof}
  The equivalence (a) $\Leftrightarrow$ (b) follows from \autoref{thm:crit-developable}.
  The equivalent (b) $ \Leftrightarrow $ (c) follows from \autoref{thm:curve-2m-lem}.
\end{proof}

Recall that $\gamma^*\bpi$ is the generalized conormal morphism
given in \autoref{thm:defn-conormal}.

\begin{lem}\label{thm:gamma-iff-gammabarpi}
  Assume that $\gamma$ is generically finite.
  Then, $\gamma$ is separable if and only if so is
  $\gamma^*\bpi: \gamma^*V_{\Gr(\mpl_{\gamma}, \PN)} \rightarrow \Pv$.
\end{lem}
\begin{proof}
  For $\Y \subset \Gr(\mpl_{\gamma}, \PN)$, the closure of the image of $\gamma$,
  we set $Y \subset \Pv$ to be the image of $V_{\Y}$ under $\bpi$.
  Since $\Y$ is of dimension one, $\bpi$ is separable and generically finite,
  due to \autoref{thm:curve-sep-uniruling}. Hence the assertion follows.
\end{proof}

Considering the dual of the above statement, we also have:

\begin{cor}\label{thm:sigma-iff-sigmapi}
  Assume that $\sigma$ is generically finite.
  Then $\sigma$ is separable if and only if so is $\sigma^*\pi: \sigma^*U_{\Gr(\mm_{\sigma}, \PN)} \rightarrow \PN$.
\end{cor}

\begin{rem}\label{thm:curve-X-sep-ruled}
  \mbox{}
  \begin{inparaenum}
  \item If $(\X,X)$ is developable and $\gamma$ is generically finite,
    then $\pi$ is birational, due to \autoref{thm:curve-sep-uniruling} and \autoref{thm:X-sep-ruled}.
    Moreover, if $\gamma$ is separable, then we have $\X = \X_0$.

  \item 
    If $X \subset \PN$ is non-degenerate and is not a cone, then it follows from \autoref{thm:nondege-notcone}
    that $\gamma$ and $\sigma$ are generically finite.
  \end{inparaenum}
\end{rem}

Recall that $\gamma^i$ and $\sigma^i$ are composite maps given in
\autoref{sec:dual-one-dimens}.
We denote by
$T X = T^{1} X := \overline{\bigcup_{x \in X^{sm}} \TT_xX} \subset \PN$, the \emph{tangent variety},
and by $T^0X := X$, $T^iX := T(T^{i-1}X)$.

\begin{thm}\label{thm:curve-gamma-sigma}
  Let $\X \subset \Gm$ and ${\X'} \subset \Gr(m', \PN)$ be projective curves
  with projections $\pi_{\X}: U_{\X} \rightarrow X$ and $\pi_{\X'}: U_{\X'} \rightarrow X'$.
  Then, for an integer $\epsilon \geq 0$, the following are equivalent:
  \begin{enumerate}
  \item
    $({\X'},X')$ is developable, 
    $\gamma^{\epsilon} = \gamma^{\epsilon}_{\X'}$ is separable,
    $\gamma^{\epsilon}{\X'} = \X$,
    and $X'$ is non-degenerate and is not a cone.
  \item\label{item:curve-gamma-sigma:b}
    $({\X},X)$ is developable,
    $\sigma^{\epsilon} = \sigma^{\epsilon}_{\X}$ is separable,
    $\sigma^{\epsilon}\X = {\X'}$,
    and $X$ is non-degenerate and is not a cone.
  \end{enumerate}
  In this case, $m = m'+\epsilon$ and $X = T^{\epsilon}X'$.
\end{thm}

\begin{proof}
  (b) $ \Rightarrow $ (a):
  It is sufficient to show the case $\epsilon = 1$.
  Since ${\X}$ is developable, it follows from \autoref{thm:curve-developability}
  that $m'=m_{\sigma}^{-}$ is equal to $m-1$.
  From \autoref{thm:nondege-notcone}, $\sigma_{\X}$ is generically finite.
  From \autoref{thm:sigma-iff-sigmapi}, $\sigma_{\X}^*\pi$ is separable.
  Applying \autoref{thm:mainthm-2-cor},
  we have that
  $\sigma_{\X} \circ \gamma_{\X'}$ gives an identity map of an open subset of ${\X'}$,
  and that ${\X'}$ is developable.
  In addition,
  $\X$ is equal to the image of the Gauss map $\gamma_{X'}$;
  hence the image of  $U_{\X} \rightarrow \PN$ is equal to $TX'$,
  which means that $X = TX'$.

  The converse (a) $ \Rightarrow $ (b) follows
  in the same way.
\end{proof}

In the statement of (a) of \autoref{thm:curve-gamma-sigma},
if $m'=0$ and $C:=\X' \subset \PN$,
then we regard $C$ itself as a developable parameter space (of $0$-planes).

We denote by $\Tan^{(i)}C$
the \emph{osculating scroll} (= \emph{osculating developable}) of order $i$ of a curve $C \subset \PN$
(see \cite[p.~76]{FP}, \cite[Definition~1.4]{Homma}, \cite[\textsection 3]{Piene1976}, for definition).
Here, $\Tan^{(1)}C = T^1C$ holds.
It is known that $\Tan^{(i)}C$ coincides with $T^{i}C$ 
if the characteristic is zero or satisfies some conditions (Homma \cite[\textsection 2]{Homma}).

\begin{cor}\label{thm:one-para=oscu}
  Assume one of the conditions \textnormal{(a)} and \textnormal{(b)} of \autoref{thm:curve-gamma-sigma},
  and assume that $m' = 0$, i.e., $C := {\X'}$ is a curve in $\PN$.
  Then the following holds:

  \begin{inparaenum}
    \setcounter{enumi}{2}

  \item 
    $C = \sigma^{m}\X$ and $\X = \gamma^{m}C$;
    in particular,
    $X = T^{m}C$.

  \item $T^{i}C = \Tan^{(i)}C$ for $0 < i \leq m+1$.

  \item 
    If $\gamma_{\X}$ is separable (equivalently, so is $\gamma_X$) and $m+1 < N$,
    then $\X$ is the closure of the space parametrizing general fibers of $\gamma_X$.
  \end{inparaenum}
\end{cor}

In the case where $X \subset \PN$ is a cone with maximal vertex $L$,
considering the linear projection from $L$
and using \autoref{thm:one-para=oscu},
we have that $X$ is a cone over an osculating scroll of order $m-\dim(L)-1$ of a certain curve
in $\PP^{N-\dim(L)-1}$
if $(\X,X)$ is developable and $\sigma^{m-\dim(L)-1}$ is separable.

\begin{proof}[Proof of \autoref{thm:one-para=oscu}]
  \begin{inparaenum}
    \setcounter{enumi}{2}

  \item The statement follows from \autoref{thm:curve-gamma-sigma};
    in particular, $\gamma^m$ is separable, $\X = \gamma^mC$, and $X = T^mC$.
    
  \item 
    For $0 \leq i < m$, it follows 
    from the diagram \ref{thm:expand-develop:diagram} of \autoref{thm:expand-develop} that
    $\gamma_{T^iC} : T^iC \dashrightarrow \gamma^{i+1}C$ is separable;
    then
    $T^iC$ is reflexive as in \autoref{thm:one-dim-ref}.
    Inductively, $T^{i+2}C = \Tan^{(i+2)}C$ follows from
    \cite[Corollary~2.3 and Theorem~3.3]{Homma}.
    
  \item If $\gamma_{\X}$ is separable, then $\X = \X_0$ as in \autoref{thm:curve-X-sep-ruled}.
  \end{inparaenum}
\end{proof}

Let us consider $X^* \subset \Pv$, the dual variety of $X \subset \PN$.
Then we have the following relation with dual varieties and tangent varieties.
\begin{cor}\label{thm:one-dim-dual-tan}
  Let $m, \epsilon$ be integers with $\epsilon > 0$ and $m+\epsilon+1 < N$,
  let $X \subset \PN$ be a non-degenerate projective variety of dimension $m+1$,
  and let $\X \subset \Gm$ be a projective curve such that $(\X,X)$ is developable.
  If $\gamma^{\epsilon+1} = \gamma^{\epsilon+1}_{\X}$ is separable,
  then we have $T^{\epsilon}((T^{\epsilon}X)^*) = X^*$ in $\Pv$.
\end{cor}

\begin{proof}
  We can assume that $X$ is not a cone.
  By definition, $X$ is the image of $U_{\X} \rightarrow \PN$.
  From \autoref{thm:curve-gamma-sigma},
  $T^{\epsilon}X$ is given by the image of
  $U_{\gamma^{\epsilon}\X} \rightarrow \PN$.
  Let $\Y := \gamma^{\epsilon+1}\X$,
  and let $Y$ be the image of $V_{\Y} \rightarrow \Pv$.
  Since $\sigma^{\epsilon} = \sigma^{\epsilon}_{\Y}$ is separable,
  considering the dual of the above statement,   we have that
  ${T^{\epsilon}Y}$ is given by the image of $V_{\sigma^{\epsilon}\Y} \rightarrow \Pv$.
  On the other hand, for each $0 \leq i \leq \epsilon$,
  from the diagram~\ref{thm:expand-develop:diagram},
  since $\gamma^{i+1}\X$ is equal to the image of
  the Gauss map $\gamma_{T^{i}X}$,
  the dual variety $(T^{i}X)^*$ is given by the image of $V_{\gamma^{i+1}\X} \rightarrow \Pv$
  (see \autoref{thm:defn-conormal}, \autoref{thm:ref-finite-proj}).
  In particular, $Y = (T^{\epsilon}X)^*$.
  From \autoref{thm:curve-gamma-sigma}, it follows that ${\sigma^{\epsilon}\Y} = \gamma^{1}\X$.
  Hence
  $T^{\epsilon}(Y)$ and $X^*$ coincide, since these
  are given by the image of $V_{\gamma^1\X} \rightarrow \Pv$.
\end{proof}

\vspace{1ex}

\end{document}